\documentclass[12pt,authoryear,letterpaper,reqno,oneside]{amsart}
\pdfoutput=1
\usepackage[left=1.4in,right=1.4in,bottom=1.4in,top=1.40in]{geometry}
\usepackage{amsmath,amsfonts,amssymb}%
\usepackage{mathrsfs} 
\usepackage[usenames]{color}
\usepackage{setspace}
\usepackage{microtype}
\usepackage{comment}

\usepackage{subfiles}

\definecolor{darkred}{rgb}{0.5,0,0}
\definecolor{darkgreen}{rgb}{0, 0.3,0}
\definecolor{darkblue}{rgb}{0,0,0.6}
\definecolor{LightGray}{rgb}{.6,.6,.6}

\usepackage[colorlinks,citecolor=darkblue,urlcolor=darkblue,linkcolor=darkgreen]{hyperref}

\usepackage[capitalize]{cleveref}  
\crefname{lemma}{Lemma}{Lemmas}
\crefname{corollary}{Corollary}{Corollaries}
\crefname{theorem}{Theorem}{Theorems}
\crefname{equation}{Equation}{Equations}
\crefname{example}{Example}{Examples}
\crefname{section}{Section}{Sections}
\crefname{subsection}{Section}{Sections}

\theoremstyle{plain}%
\newtheorem{theorem}{Theorem}[section]%

\newtheorem{proposition}[theorem]{Proposition}%
\newtheorem{lemma}[theorem]{Lemma}%
\newtheorem{corollary}[theorem]{Corollary}%
\newtheorem{definition}[theorem]{Definition}%

\newcommand{\Lang}{\EM{\mathcal{L}}}

\newcommand{\SymDiff}{\EM{\triangle}}

\renewcommand{\hat}{\widehat}

\def\st{\,:\,}

\def\M{{\EM{\mathcal{M}}}}

\def\N{{\EM{\mathcal{N}}}}

\def\Pr{{\EM{\mathbb{P}}}}

\def\H{{\EM{\mathcal{H}}}}

\def\F{{\EM{\mc{F}}}}

\def\x{{\EM{\ol{x}}}}

\def\y{{\EM{\ol{y}}}}

\def\a{{\EM{\ol{a}}}}

\def\b{{\EM{\ol{b}}}}

\def\c{{\EM{\ol{c}}}}
\def\d{{\EM{\ol{d}}}}

\def\kk{{\EM{\ol{k}}}}
\def\mbar{{\EM{\ol{m}}}}

\newcommand{\id}{\ensuremath{\mathrm{id}}}

\def\L{{\EM{\mc{L}}}}

\newcommand{\defn}[1]{{\bf{#1}}}
\newcommand{\defas}{{\EM{\ :=\ }}}

\newcommand\len{\mathrm{len}}
\newcommand\range{\mathrm{range}}

\def\ordp{{\EM{\ol{\tau}}}}
\def\brap{{\EM{\widehat{\tau}}}}

\def\tind{{\EM{t_{\mathrm{ind}}}}}
\def\ind{{\EM{\mathrm{ind}}}}

\def\dist{{\EM{d_{1}}}}

\newcommand{\dee}{\mathrm{d}}
\newcommand{\boldx}{\mathbf{x}}
\newcommand{\Ind}{\mathbf{1}}

\def\Borel#1{\EM{\ol{#1}}}

\def\floor#1{\EM{\lfloor #1 \rfloor}}

\def\Nats{{\EM{{\mbb{N}}}}}

\def\^{\EM{{}^{\And}}}

\def\And{\EM{\wedge}}

\def\<{\EM{\langle}}
\def\>{\EM{\rangle}}

\def\nl{\newline}

\def\EM#1{\ensuremath{#1}}
\def\mbb#1{\EM{\mathbb{#1}}}
\def\mbf#1{\EM{\mathop{\pmb{#1}}}}
\def\mc#1{\EM{\mathcal{#1}}}

\def\ol#1{\EM{\overline{#1}}}

\def\ul#1{\underline{#1}}

\newcommand{\arity}{\EM{\mathrm{arity}}}

\newcommand{\Powerset}{\raisebox{.15\baselineskip}{\Large\ensuremath{\wp}}}

\definecolor{MyGreen}{rgb}{.75,0,.75}
\definecolor{RealGreen}{rgb}{0,1,0}
\definecolor{DarkGreen}{rgb}{0.1,0.4,0.1}
\definecolor{ActualGreen}{rgb}{0.0,0.5,0.0}
\definecolor{MyBlue}{rgb}{0,0,1}
\definecolor{MyRed}{rgb}{1,0,0}

\definecolor{darkred}{rgb}{0.5,0,0}
\definecolor{darkgreen}{rgb}{0, 0.3,0}
\definecolor{darkblue}{rgb}{0,0,0.6}

\definecolor{LightGray}{rgb}{.6,.6,.6}

\begin{document}


\title[Stable regularity for relational structures]
{Stable regularity for relational structures}

\author[Ackerman]{Nathanael Ackerman}
\address{
Harvard University\\
Cambridge, MA 02138\\
USA}
\email{nate@math.harvard.edu}

\author[Freer]{Cameron Freer}
\address{
	Remine\\
	Fairfax, VA 22031\\
USA
}
\email{cameron@remine.com}

\author[Patel]{Rehana Patel}
\address{
	Wheaton College\\ Norton, MA 02766\\ USA
}
\email{rrpatel@cantab.net}


\begin{abstract}
	We generalize the stable graph regularity lemma of Malliaris and Shelah to the case of finite structures in finite relational languages, e.g., finite hypergraphs.
	We show that under the model-theoretic assumption of stability, such a structure
	has an equitable regularity partition 
	of size polynomial in the reciprocal of the desired accuracy,
	and 
	such that for each
	$k$-ary relation
	and $k$-tuple of elements of the partition, the density
	is close to either $0$ or $1$.
	In addition, we provide regularity results for 
	finite and Borel structures that satisfy a weaker notion that we call \emph{almost stability}.
\end{abstract}

\maketitle


\setcounter{page}{1}
\thispagestyle{empty}

\begin{small}
\renewcommand\contentsname{\!\!\!\!}
\setcounter{tocdepth}{3}
\tableofcontents
\end{small}



\section{Introduction}
	
	Szemer\'edi's regularity lemma for graphs is a fundamental tool in combinatorics. It can be viewed as saying that every finite graph can be approximated by one that has a small ``structural skeleton'' overlaid with randomness. 
	Malliaris and Shelah \cite{MS} show that one can obtain more control over this approximation in the presence of a model-theoretic tameness condition known as \emph{stability}, that is essentially combinatorial in nature. In this paper, we extend the result of Malliaris and Shelah to the case of arbitrary finite structures in a finite relational language. In particular, our result yields better bounds on hypergraph regularity approximations in the presence of stability.
	
	The Szemer\'edi regularity lemma 
can be expressed more formally as saying that for any finite graph there is a partition of the vertices,
	known as a \emph{regularity
	partition}, such that the partition is 
\emph{equitable} (i.e., the sizes of the parts differ by at most $1$), and
	for all but a few 
	pairs 
	of (not necessarily distinct) elements of the partition,
	the induced subgraph on the vertices among that pair is
close to
a random bipartite graph (or random graph, if the parts are not distinct) 
having some edge density between $0$ and $1$.
The pairs for which this does not hold are called \emph{irregular}.
The accuracy of the approximation yielded by a regularity partition
	is measured both in terms of having few irregular pairs, and by the
	closeness of 
	each regular pair to a random (bipartite) graph. 
	The regularity lemma provides
	an upper bound on the size of a regularity partition that depends only on the desired accuracy of the approximation,
	and not on the particular graph being approximated.
	For details, see, e.g., \cite{MR2798368}.

While this bound on the size of the regularity partition depends only on the desired accuracy, in general one cannot guarantee a bound better than a tower of exponentials (of height that is polynomial in the reciprocal of the accuracy) 
	\cite{MR1445389}.
Further, it has long been known that if a graph contains a large \emph{half-graph} as an induced subgraph, then any regularity partition for the graph must have irregular pairs (independently observed by Lov\'asz, Seymour, and Trotter
and by
Alon, Duke, Leffman, R\"odl, and Yuster \cite{MR1251840}).

	Malliaris and Shelah \cite{MS} observed that the presence of a large induced half-graph corresponds to the absence of \emph{stability},
	a key property from model theory
	that provides a sense in which a combinatorial object 
	is highly structured, or tame
	(for details, see \cite{Classification-theory}). 
	Malliaris and Shelah \cite{MS}
	show that when a graph is stable, it admits a regularity partition with no irregular pairs, with a number of parts that is merely polynomial in the reciprocal of the accuracy, and where 
	for each pair of (not necessarily distinct) parts, 
	the induced bipartite graph 
	across the parts
	(or induced graph on the one part) 
	is either complete or empty.
	In other words, this polynomial-size partition of the vertices is such that for every pair $(V_1, V_2)$ of elements of the partition (possibly with $V_1 = V_2$), the induced subgraph on $V_1 \cup V_2$
	can be modified by a small number of edges so that either between every pair of distinct elements, one from $V_1$ and the other from $V_2$, there is an edge, or between every pair of distinct elements, one from $V_1$ and the other from $V_2$, there is no edge.
	In this case, the graph is close in edit distance to an \emph{equitable blow-up} of a small finite graph (possibly with self-loops).

	The regularity lemma for graphs has been generalized to finite structures in a finite relational language (see, e.g., \cite{2014arXiv1412.8084A}), a key case of which are the $k$-uniform hypergraphs (see, e.g.,
	\cite{MR2212136},
	\cite{MR2373376}, 
	\cite{MR2351688}, 
	and \cite{MR2964622}).
	The upper bounds on the partition size are even worse than for graphs, as Moshkovitz and Shapira have recently shown that the bounds are necessarily of Ackermann-type.
	The model-theoretic notion of stability also makes sense in the context of finite relational languages. In this paper, we extend Malliaris and Shelah's results to show that every finite stable structure in a finite relational language admits an equitable partition with polynomially many parts such that for every relation $R$ (of arity $k$, say) and every $k$-tuple $(V_1, \ldots, V_k)$ of parts (possibly with repetition),
	the induced substructure restricted to $R$
	on $V_1 \cup \cdots \cup V_k$ can be modified by a small number of ``$R$-edges'' so that either
	every $k$-tuple of elements
	in $V_1 \times \cdots \times V_k$
forms an $R$-edge, or
	every $k$-tuple of elements 
	in $V_1 \times \cdots \times V_k$
	does not form an $R$-edge. 
	In particular, the relational structure is close in edit distance to an equitable blow-up of a small structure in the same language. 
	This shows that in the stable case, not only is ``randomness'' in the $R$-edges eliminated in the approximation, but so are the ``intermediate levels'' that are a key complication of the general case of 
hypergraph regularity lemmas.
	Our proof closely follows the methods of \cite{MS}.

	In the case of finite
	relational structures  
	that are \emph{almost stable} 
	(in a sense that we make precise), we again show that the structure is close in edit distance
to an equitable blow-up of a small finite structure, albeit where the few edits may not be distributed  as uniformly as we can require in the stable case.
	Finally, we provide a similar regularity lemma for almost stable relational structures that are Borel.

	\subsection{Related work}
Expanding on Malliaris and Shelah's stable regularity lemma for graphs,
Malliaris and Pillay \cite{MR3451251} give a short proof of the stable regularity lemma for arbitrary Keisler measures. 
In this more general setting, they obtain most of the nice properties from the stable regularity lemma on graphs \cite{MS}, but they do not get precise bounds on the size of the partition.

Independently from our work in the present paper, Chernikov and Starchenko 
\cite{2016arXiv160707701C}
prove a stable regularity lemma for Keisler measures over finite and Borel structures in a language with a single relation. In the case of finite structures, their stable regularity lemma is closely related to our main result,  \cref{Theorem: Stable Regularity for Relational Structures}, restricted to languages with a single relation. However, while the partitions they obtain are definable (unlike ours), they need not be equitable. 

Chernikov and Starchenko also obtain two regularity lemmas for structures satisfying certain model-theoretic conditions other than stability, one for NIP structures
that generalizes a result of Lov\'asz and Szegedy \cite{MR2815610},
and one for distal structures, generalizing their earlier result \cite{2015arXiv150701482C}.

Generalizing Green's group-theoretic regularity lemma \cite{MR2153903}, Terry and Wolf
obtain a stable version for vector spaces over finite fields
\cite{2017arXiv171002021T}, and
Conant, Pillay, and Terry 
obtain a further generalization to arbitrary finite groups
\cite{2017arXiv171006309C}.

\subsection{Road map of the proof of the main result}

Before beginning our technical construction, we 
here provide a road map of the proof of the main result,
\cref{Theorem: Stable Regularity for Relational Structures}.
We will first describe how to ``augment'' relations and give a quick proof outline in terms of such augmented relations. Then we will provide more detail on three key aspects: obtaining $\varepsilon$-excellent sets, making a partition equitable, and modifying the original structure so that it is a blow-up.

Let $\Lang$ be a finite relational language, and let $\brap\in\Nats$. Suppose that $\M$ is a finite $\Lang$-structure such that none of its relations has the so-called $\brap$-branching property. (In fact, a slightly weaker hypothesis will suffice.)
In particular, $\M$ is stable.

We begin by augmenting every relation in $\M$. Each relation in $\M$ can be thought of as a $\{\top, \bot\}$-valued function of some arity. We replace each relation with a continuum-sized family of functions (indexed by $\varepsilon>0$) each of which takes values in $\{\top, \bot, \uparrow\}$, and further allow each argument to be either an element or a subset of $\M$.
In the case where exactly one argument is a subset of $\M$,
this will be done by ``polling'' the elements in a subset and assigning a truth value ($\top$ or $\bot$) if and only if a sufficiently large majority (namely, a ($1-\varepsilon$)-fraction) of the elements agree on that truth value (when all other arguments are fixed), and $\uparrow$ otherwise. However, when more than one argument is a subset, the polling is more complicated.
For a given order of arguments, we will define this notion of polling by induction on the number of arguments that are sets, in a way that depends on the order of arguments polled so far. 

These augmented relations will be used to construct
collections of so-called
\emph{$\varepsilon$-excellent sets,} that in particular are such that whenever all arguments of an augmented relation are $\varepsilon$-excellent then the (function indexed by $\varepsilon$ of the) augmented relation has a truth value (i.e., is assigned $\top$ or $\bot$).  

The proof outline is as follows.
Assume that $\M$ is large enough (relative to $\brap$). We first find, using the augmented relations, 
an $\varepsilon$-excellent partition of a large subset of $M$, the underlying set of $\M$. We then 
transform this into an equitable partition of $M$ into $(\varepsilon+\zeta)$-excellent sets
(where $\zeta$ depends only on $\varepsilon$).
Finally, we show that it is possible to change some $\varepsilon$-fraction of the (original) relations so that an equitable partition now describes this modification of $\M$ as exactly the
blow-up of a small finite structure, whose size (i.e., the number of elements of an equitable partition) is at most polynomial in $\varepsilon$, where the polynomial's exponent depends only on $\brap$ and the maximum arity of $\Lang$.

\subsubsection{$\varepsilon$-excellent sets}
Suppose $A\subseteq \M$. We now describe how to find an $\varepsilon$-excellent subset of $A$ that is 
\emph{big} in the sense that its size is among a particular collection of natural numbers determined by $\varepsilon$.
We show that a witness
to the non-$\varepsilon$-excellence of $A$ can be taken to consist of a relation $R$, an order of its arguments, an index $j$ among the $\arity(R)$-many arguments, an $(\arity(R)-1)$-tuple of sets $\<B_i\>_{i\neq j}$ 
(satisfying a certain additional property with respect to the order)
and two big
disjoint subsets 
$A_0$ and $A_1$,
such that the truth value assigned by
the augmentation of $R$ (with polling based on the given ordering) 
to $\<B_i\>_{i\neq j}$ along with $A_0$ in the $j$th coordinate is different from
the truth value that it assigns to $\<B_i\>_{i\neq j}$ along with $A_1$ in the $j$th coordinate.
Having found such a witness to the non-$\varepsilon$-excellence of $A$, we then
look for such witnesses to the non-$\varepsilon$-excellence of $A_0$ and of $A_1$.
We repeat this process on big disjoint subsets of $A_0$ and of $A_1$, etc., and stop as soon as some branch can go no farther (because we have reached some big subset of $A$ that itself has no such witness), after which the resulting binary tree of subsets of $A$ is perfect.
A \emph{mesa} is an object of the following sort that arises from 
a perfect tree of such witnesses: a finite perfect binary tree, each node of which is labeled by 
a triple consisting of a relation symbol, an index for one of the arguments of the relation, an ordering for the arguments of the relation, and certain witnessing subsets.
At least one node of a maximal mesa does not itself have witnesses;
we call such a node a
\emph{cap},
and it turns out that the height of any maximal mesa can be 
bounded above
in terms of $\brap$.
The intuitive idea is that a mesa is not too ``tall'', by virtue of not being too ``wide''; there can be many caps on it --- by virtue of any of which it doesn't get too ``tall''.

Mesas
have
three important properties. 
First, as already mentioned, each chosen subset of $A$ occurring in its tree 
is big (i.e., its size is in the special set of sizes).
Second, also as already noted, if the mesa is maximal,
then there must be at least one cap,
whose corresponding subset must therefore be $\varepsilon$-excellent. 
Third, from any mesa such that every node has the same labels for the relation, argument index, and argument order, we can extract a witness to the branching property of $\M$ of the same height as the mesa.

Next, by a Ramsey-theoretic result,
there is a function $f\colon \Nats \to \Nats$ such that $f= O(n \log n)$ with the following property: whenever $k\in\Nats$ and
$T$ is a perfect binary tree with height $f(k)$, each node of which is labeled by a triple
consisting of a
relation symbol, an index for one of the arguments of the relation, and an ordering
for the arguments of the relation, there is a perfect subtree of $T$ of height $k$
such that every node of the subtree has the same label.
In particular, this holds of a mesa. Hence from a bound on the branching property for $\M$ we may obtain a bound on the height of any mesa arising from $\M$.

Because we have bounds on how much the sets decrease in size as one proceeds down a mesa, the bound on the height of the mesa induces a bound on the size of the excellent sets. In aggregate, using the fact
that no relation has the $\brap$-branching property,
we can find
a constant $c_{\varepsilon}$ such that any set $A$ has an $\varepsilon$-excellent set of size at least $c_{\varepsilon}\cdot |A|$. 

\subsubsection{Equitable partitions}
We now describe in more detail how we
find an equitable partition of ``most'' of $\M$ consisting of $(\varepsilon+\zeta)$-excellent sets. 
Using the
method for extracting excellent subsets that have size at least a positive fraction, we repeat this procedure to get a partition of ``most'' of the structure where every element of the partition is excellent and the size of the partition is bounded 
in terms of $\varepsilon$.
We then aim to modify this partition to an equitable one while only increasing the error slightly.
The allowable sizes for a ``big'' set in fact were chosen so that their greatest common denominator is also in the set. 
Consider a random, equitable, refinement of the original partition where the size of each element is this greatest common divisor. Using the fact that all relations of $\M$ are appropriately stable, 
the limiting properties of certain hypergeometric distributions imply that
with high probability a random such partition is $(\varepsilon + \zeta)$-excellent provided that the structure underlying the partition is ``large''.
In particular, this implies that there is some such equitable $(\varepsilon+ \zeta)$-refinement.

\subsubsection{Modifying the original structure}
We now describe how to change the truth values of each relation on just
an $(\varepsilon \cdot r)$-fraction
of the elements (where $r$ is the arity of the relation),
so that the resulting structure is the 
blow-up of a finite structure of size bounded by a polynomial in $\varepsilon^{-1}$. 
This modification of the structure has two parts. 
First, we show that for any $\varepsilon$-excellent partition of ``most'' of $\M$, the relations may be modified on a
small portion of the elements so as to obtain a
partition of the same set which is ``indiscernible'' (i.e., a blow-up of a finite structure). 
Next we have to deal with the (small number of) elements of $\M$ not in any part of the original partition. 
We show that if we add such elements to parts of the partition arbitrarily (while keeping the partition equitable), we may then modify relations on these elements (with respect to the other elements) so that in the modified structure the relations agree with the other elements within the part to which they were assigned. In aggregate these actions only require us to change the relations on a small fraction of the elements, yielding a structure that is exactly a blow-up while being close to the original.

\subsection{Notation}

We now introduce some notation and conventions that we will use throughout the paper. 

All logarithms are in base 2, and are denoted by $\log$ (with no subscript).

In this paper, $\Lang$ denotes a fixed finite relational language. All $\Lang$-formulas are first-order. We consider equality to be a logical symbol and not a member of \Lang. 

For any relation $E \in \Lang$, let $\arity(E)$ denote the arity of $E$. We will also need the following two quantities related to the arities of relations in \Lang;  let
\[
    q_\Lang \defas \max\{\arity(E) \st E \in \Lang\}
\]
and 
\[
	n_\Lang \defas |\Lang| \cdot q_\Lang. 
	\]

We consider an $n$-element sequence $\a$ of elements of $A$ to be a map of the form $\a\colon \{0, \ldots n-1\} \to A$, and therefore $\emptyset$ is the empty sequence, 
and $\range(\a)$ is the set of elements occurring in the sequence $\a$.
We also 
write $\len(\a) = n$ for the length of such a sequence, and
identify 
$\a$ with the tuple of its elements $\<\a(0), \ldots, \a(n-1)\>$.

For finite tuples
$\<a_i\>_{i < n}$ and $\<b_j\>_{j < m}$, we say that $\<a_i\>_{i < n}$ is an \emph{initial segment} of $\<b_j\>_{j < m}$, written
\[ 
    \<a_i\>_{i < n} \preceq \<a_j\>_{j < m},
\]
when $n \leq m$ and when $a_i = b_i$ for all $i < n$. 
Given a tuple $\a = \<a_0, \ldots, a_{n-1}\>$ and an element $b$, we 
write $\a\^b$ to
denote the tuple $\<a_0, \ldots , a_{n-1}, b\>$.

We refer to the elements of a partition as its \emph{parts}.
We now introduce two special kinds of partitions.  
An \emph{equitable} partition is one whose parts differ in size by at most 1. 

\begin{definition}
Suppose $\M$ is an $\Lang$-structure with underlying set $M$. We say that $P$ is a partition of $\M$ if it is a partition of $M$. We say that $P$ is \defn{equitable} if for any $p_o, p_1 \in P$, 
\[
    \bigl | \, |p_0| - |p_1| \,\bigr| \leq 1.
\]
\end{definition}

An \emph{indivisible} partition of an $\Lang$-structure is one for which, given any tuple,  whether or not a relation holds of the tuple depends only on which respective parts of the partition the elements of the tuple are in.

\begin{definition}
We say that a partition $P$ of an $\Lang$-structure $\M$ is \defn{indivisible} if for each relation $E\in\Lang$, for all 
$p_0, \dots, p_{\arity(E)-1} \in P$,
and for any pair of tuples $\<a_i^0\>_{i < \arity(E)}, \<a_i^1\>_{i < \arity(E)}$ such that $a_i^0, a_i^1 \in p_i$, where $0 \le i < \arity(E)$, we have
\[
\M \models E(a_0^0, \dots, a_{\arity(E)-1}^0) \leftrightarrow E(a_0^1, \dots, a_{\arity(E)-1}^1).
\]
\end{definition}

Note that a  partition of an $\Lang$-structure is indivisible when we can 
obtain an $\Lang$-structure
by quotienting 
out 
by 
the equivalence relation induced by the partition.

\begin{definition}
	Suppose $\M$ and $\N$ are $\Lang$-structures with underlying sets $M$ and $N$ respectively. A map $\alpha\colon M \to N$ is a \defn{full homomorphism} from $\M$ to $\N$ if for each relation $E \in \Lang$ and all tuples $a_0, \dots, a_{\arity(E)-1} \in M$ of 
	elements of $M$, 
\[
\M \models E(a_0, \dots, a_{\arity(E)-1}) \text{ if and only if } \N \models E(\alpha(a_0), \dots, \alpha(a_{\arity(E)-1})).
\]
\end{definition}
Note that full homomorphisms are not necessarily injective.

\begin{definition}
	\label{gen-blowup}
An $\Lang$-structure $\M$ is a \defn{blow-up} of an $\Lang$-structure $\N$
when there is a surjective full homomorphism $i\colon \M \to \N$. We call $i$ the \defn{witness} to the blow-up. 

If further the sets
	$i^{-1}(\{b_0\})$ and $i^{-1}(\{b_1\})$ differ in size by at most one,
	for all $b_0, b_1 \in \N$, 
	then $\M$ is an \defn{equitable blow-up} of $\N$. 
\end{definition}

The regularity lemmas that we obtain in this paper can be seen as stating that certain types of structures
are close in edit distance to a blow-up of a small finite structure.

The following easy lemma, whose proof we omit, makes precise the notion that an $\Lang$-structure with an indivisible partition can be thought of as blow-up of a smaller $\Lang$-structure. 

\begin{lemma}
For an $\Lang$-structure $\M$ and a partition $P$ of $\M$ the following are equivalent.

	\begin{itemize}
\item $P$ is indivisible. 

\item There exists an  $\Lang$-structure $\N$ such that $\M$ is a blow-up of $\N$ with witness $i$ such that 
\[
P = \{i^{-1}(\{b\})\st b \in N\}. 
\]
	\end{itemize}

Furthermore, $\M$ is an equitable blow-up of $\N$ if and only if $P$ is equitable. 
\end{lemma}

Intuitively, $\M$ is a blow-up of $\N$ if it can be obtained by replacing each element of $\N$ with an indiscernible set, while $\M$ is an equitable blow-up of $\N$ if these indiscernible sets are all almost the same size.

\subsection{Stability}

We now recall some basic definitions and facts from stability theory, following the exposition in
Malliaris and Shelah \cite{MS}.

\begin{definition}
Let $\ordp \in \Nats$. An $\Lang$-formula $\varphi(\x; \y)$ has the \defn{$\ordp$-order property} in an $\Lang$-structure $\M$ when there exist tuples $\<\a_i\>_{i<\ordp} \subseteq \M$ (with $\len(\a_i)=\len(\x$ for all $i < \ordp$) and $\<\b_j\>_{j<\ordp} \subseteq \M$ (with $\len(\b_j) = \len(\y)$ for all $j < \ordp$) such that for all $i, j < \ordp$,
\[
    \M\models \varphi(\a_i; \b_j) \Leftrightarrow i < j.
\]

We say that $\varphi(\x;\y)$ has the \defn{non-$\ordp$-order property} in $\M$ when it does not have the $\ordp$-order property in $\M$. 
\end{definition}

Note that the $\ordp$-order property is defined 
for a formula along with a given partition of its free variables, not just for the formula alone.

We will in fact work with a combinatorial property that holds in a structure essentially whenever the $\ordp$-order property does.

\begin{definition}
Let $\brap \in \Nats$. An $\Lang$-formula $\varphi(\x; \y)$ has the \defn{$\brap$-branching property} in an $\Lang$-structure $\M$ when there exist tuples $\<\a_i\>_{i\in \{0,1\}^\brap} \subseteq \M$ (with $\len(\a_i)=\len(\x)$ for all $i \in \{0,1\}^\brap$) and $\<\b_j\>_{j \in \{0,1\}^{<\brap}} \subseteq \M$ (with $\len(\b_j) = \len(\y)$ for $j\in \{0,1\}^{< \brap}$)
such that for all $i\in \{0,1\}^\brap$, for all $j \in \{0,1\}^{<\brap}$, and for each $h \in \{0, 1\}$, we have that
\[
    j\^h \preceq i
\]
implies
\[
    \M\models \varphi(\a_i; \b_j) \Leftrightarrow (h = 1).
\]

We say that $\varphi(\x;\y)$ has the \defn{non-$\brap$-branching property} in $\M$ when it does not have the $\brap$-branching property in $\M$. 
\end{definition}

We now state a connection between the non-$\ordp$-order property and the non-$\brap$-branching property for a structure $\M$.

\begin{lemma}[{\cite[Lemma 6.7.9]{MR1221741}}]
	\label{Lemma: Branching property from order property}
If $\varphi(\x; \y)$ has the non-$\ordp$-order property in $\M$ then $\varphi(\x; \y)$ has the non-$2^{\brap}$-branching property in $\M$, where $\brap=2^{\ordp+2}-2$. On the other hand, if $\varphi(\x;\y)$ has the non-$\brap$-branching property in $\M$ then $\varphi(\x;\y)$ has the non-$2^{\ordp}$-order property in $\M$, where $\ordp = 2^{\brap+1}$.
\end{lemma}

While we have defined the order and branching properties for arbitrary formulas and partitions of their variables, we 
will focus on
the situation where these formulas are relation symbols of $\Lang$.

\begin{definition}
Let $\M$ be an $\Lang$-structure. We say that $\M$ has the 
	\defn{non-$\ordp$-order property}
	(\defn{non-$\brap$-branching property})
	if for each 
relation $E \in \Lang$ and each $0\leq j < \arity(E)-1$, the formula $E(x_0, \dots, x_{\arity(E)-1})$ with the partition of variables $(x_j; x_0, \dots, x_{j-1}, x_{j+1}, \dots, x_{\arity(E)-1})$ has the 
	non-$\ordp$-order property
	(non-$\brap$-branching property)
	in $\M$. 
\end{definition}

We will be interested in the case where 
$\M$ has the non-$\ordp$-order property for some $\ordp\in\Nats$,
and will work in the case where $\M$ has the non-$\brap$-branching property
for a corresponding $\brap$.

For the rest of this paper, fix $\brap \in \Nats$.

\section{Excellence}

From now on, let $\M$ be a finite $\Lang$-structure with underlying set $M$. We will prove our regularity lemmas by showing that under appropriate stability assumptions we can find,
for any subset $A$ of $M$,
a partition of $A$ with respect to which the induced substructure on $A$
is ``almost'' a blow-up. 
To do this, 
we use a notion called \emph{$\varepsilon$-excellence}, 
generalizing the definition from Malliaris and Shelah \cite{MS}, 
which captures this idea of being almost a blow-up. 

We begin by allowing relations to hold both of elements and subsets of $M$.
Let $\hat{M} \defas M \cup \Powerset(M)$ where $\Powerset(M)$ denotes the power set of $M$. We now define how to augment a relation on $M$ to be on all of $\hat{M}$ (for a given tolerance $\varepsilon$).

Write 
	$\top$ and $\bot$ to denote the ``truth values'' true and false, respectively, and $\uparrow$ for an ``indeterminate'' value.

\begin{definition}
Let $0 < \varepsilon < \frac{1}{2}$, let $E \in \Lang$ of arity $n$, and let $\mbar$ be a tuple of distinct elements of $\{0, \dots, n-1\}$. Define, inductively on the length of $\mbar$, the collection of \defn{$\varepsilon$-partial relations for $E$}. Each such partial relation is a 
	function parametrized by $\mbar$ and $\varepsilon$, of the form $\hat{E}_{\varepsilon}^{\mbar}\colon \hat{M}^n \to \{\top, \bot, \uparrow\}$.

Let $A_0, \dots, A_{n-1} \in \hat{M}$ and let $S \defas \{i < n \st A_i \in \hat{M} \setminus M\}$.
If $S \neq \range(\mbar)$,
then define
\[
\hat{E}_{\varepsilon}^\mbar(A_0, \dots, A_{n-1}) \defas \uparrow.
\]

Otherwise, when $S = \range(\mbar)$, we will define $\hat{E}_{\varepsilon}^\mbar(A_0, \dots, A_{n-1})$ by induction on $\ell \defas \len(\mbar)$, as follows. \nl\nl
\ul{Case $\ell = 0$:}
In this case, $\mbar = \emptyset$, and so $S = \emptyset$.
	In particular, $A_0, \ldots, A_{n-1}$ are elements of $M$.
	Define
\begin{itemize}
\item $\hat{E}^\emptyset_{\varepsilon}(A_0, \dots, A_{n-1}) \defas \top$ \quad if \quad $M \models E(A_0, \dots, A_{n-1})$, and 
\item $\hat{E}^\emptyset_{\varepsilon}(A_0, \dots, A_{n-1}) \defas \bot$ \quad if \quad $M \models \neg E(A_0, \dots, A_{n-1})$.
\end{itemize}

\ \nl
\ul{Case $\ell \ge 1$:}\nl
	Let $\kk$ be the initial subtuple of $\mbar$ of length $\ell - 1$, and let $j \defas \mbar(\ell-1)$ be the last element of $\mbar$, so that $\mbar = \kk \^j$.
	Because $\mbar$ is a tuple of distinct elements,
	observe that 
$\kk\colon \{0, \dots, \ell-2\} \to S \setminus \{j\}$ is a bijection. For each $\delta \in \{\top, \bot\}$, define
\[
A_{\mbar}^\delta \defas \{a \in A_j \st \hat{E}^\kk_{\varepsilon}(A_0, \dots, A_{j-1}, a, A_{j+1}, \dots, A_{n-1}) = \delta\}.
\]

\begin{itemize}
\item If 
$\displaystyle \frac{|A_{\mbar}^\top|}{|A_j|} > 1- \varepsilon$
then 
define $\hat{E}^{\mbar}_{\varepsilon}(A_0, \dots, A_{n-1}) \defas \top$.
		\vspace*{5pt}

\item If 
$\displaystyle \frac{|A_{\mbar}^\bot|}{|A_j|} > 1- \varepsilon$
then  define
$\hat{E}^{\mbar}_{\varepsilon}(A_0, \dots, A_{n-1}) \defas \bot.$
		\vspace*{5pt}

\item Otherwise define $\hat{E}^{\mbar}_{\varepsilon}(A_0, \dots, A_{n-1}) \defas  \uparrow$.
\end{itemize}
\end{definition}

Note that the last three bullet points are mutually exclusive as $\varepsilon < \frac12$.

To illustrate this definition, we walk through the cases where $|S| \leq 2$ and $\range(\mbar) = S$. 
Recall that $|S|$ is the number of arguments of $\hat{E}^{\mbar}_{\varepsilon}$ that are subsets of $M$. 
First consider the case where $|S| = 0$. We then have $\mbar = \emptyset$, and all $A_i$ are elements of $M$, and so we let $\hat{E}^\mbar_\varepsilon$ agree with the 
relation $E$ on $\<A_0, \ldots, A_{n-1} \>$. 

Next consider the case where $|S| =1$, with say $S = \{j\}$, i.e., when there is a unique element $A_j$ of $\hat{M}\setminus M$ among the arguments $A_0, \dots, A_{n-1}$. In this case we let $\hat{E}^{\<j\>}_\varepsilon(A_0, \dots, A_{n-1})$ be $\top$ if, when we fix $A_1, \dots, A_{j-1}, A_{j+1}, \dots, A_{n-1}$ and let the $j$th entry vary among the elements of $A_j$, at least a ($1-\varepsilon$)-fraction of the elements return a value of $\top$; and similarly for $\bot$. If this does not occur, i.e., if there is no ``near-consensus'' among the elements of $A_j$, then we return $\uparrow$ signifying that its value is indeterminate. 

Finally consider the case when $|S| = 2$, with say $S = \{p, q\}$. Suppose we have defined $\hat{E}^\kk_\varepsilon$ whenever $|\range(\kk)| = 1$. In other words, we have already defined both $\hat{E}^{\<p\>}_\varepsilon$ and $\hat{E}^{\<q\>}_\varepsilon$.
We would like to perform a similar sort of consensus-gathering to determine the values of 
$\hat{E}^{\<p,q\>}_\varepsilon$ and
$\hat{E}^{\<q,p\>}_\varepsilon$.
In the first case, replace $A_q$ by an element of $A_q$, and see if there is a near-consensus as this element varies within $A_q$, using the previously-defined $\hat{E}^{\<p\>}_\varepsilon$. In the second case, replace $A_p$ by an element of $A_p$, and likewise see if there is a near-consensus as it varies, using $\hat{E}^{\<q\>}_\varepsilon$.

Note that when there are at least two sets from $\hat{M}\setminus M$ among the arguments $A_0, \dots, A_{n-1}$, the order in which they are considered in the inductive definition matters (and indeed the superscript $\kk$ of $\hat{E}^\kk_\varepsilon$ keeps track of this order).
As we will see, we will mainly be interested in elements of $\hat{M}$ that have a property called \emph{$\varepsilon$-excellence}, which implies that the same truth value is returned no matter in which order we consider the arguments (i.e., where $\hat{E}^\kk_\varepsilon$ depends only on $\range(\kk)$ and not on the order in which the entries occur).

In order to define the notion of $\varepsilon$-excellence, we first need to define a notion of \emph{$(\varepsilon, \ell, E)$-goodness} for elements of $\hat{M}$, where $E \in \L$ and 
$0 \le \ell \le \arity(E)$.

\begin{definition}
Let $\varepsilon > 0$, let $E \in \Lang$ of arity $n$, and let $\ell \le n$. Define the notion of \defn{$(\varepsilon, \ell, E)$-goodness} for an element $A_0 \in \hat{M}$ by induction on $\ell$ as follows. \nl\nl
\ul{Case $\ell = 0$:}\nl
$A_0 \in \hat{M}$ is $(\varepsilon, 0, E)$-good if and only if $A_0 \in M$. \nl\nl
\ul{Case $\ell \ge 1$:}\nl
$A_0 \in \hat{M}$ is $(\varepsilon, \ell, E)$-good if and only if $A_0$ is $(\varepsilon, k, E)$-good for $1 \leq k < \ell$ and for all
\begin{itemize}
\item $A_1, \dots, A_{\ell-1} \in \hat{M}\setminus M$ such that 
		$A_i$ is $(\varepsilon, \ell-i, E)$-good
	for every $1\leq i < \ell$, 
\item $A_{\ell}, \dots, A_{n-1} \in M$, and
\item permutations $\sigma$ of $\{0, \ldots ,n-1\}$,
\end{itemize}
we have
\[
	\hat{E}^{\<\sigma(\ell-1), \dots, \sigma(1), \sigma(0)\>}_\varepsilon(A_{\sigma(0)}, A_{\sigma(1)}, \dots, A_{\sigma(n-1)}) \in \{\top, \bot\}.
\]

We say that $A\in\hat{M}$ is \defn{$\varepsilon$-excellent} when $A$ is $(\varepsilon, \arity(E), E)$-good
	for all relation symbols $E \in \Lang$.
\end{definition}

Note that 
in the case where $\M$ is a (symmetric) graph with edge relation $E$, 
our notion of $(\varepsilon,1, E)$-goodness is the same as
$\varepsilon$-goodness in \cite{MS}. Our more general 
definitions
allow us to
generalize their proof to arbitrary finite relational languages.

Again we illustrate
the cases $\ell = 1$ and $2$. First, $A_0\in\hat{M}\setminus M$ is $(\varepsilon, 1, E)$-good when $\hat{E}^\mbar_\varepsilon$ returns a truth value on any collection of arguments such that $A_0$ is the only argument from $\hat{M} \setminus M$.

Next,
$A_0\in \hat{M}\setminus M$ is $(\varepsilon, 2, E)$-good if it is $(\varepsilon, 1, E)$-good and further, for all $(\varepsilon, 1, E)$-good $A_1$, any $\varepsilon$-partial relation whose only arguments from $\hat{M}\setminus M$ are $A_0$ and $A_1$ returns a truth value when we first
vary the elements of $A_0$ and then vary the elements of $A_1$
(and this holds no matter where $A_0$ and $A_1$ occur as arguments in the relation). 

The notion of $(\varepsilon, \ell, E)$-goodness generalizes this idea. For $\ell\ge 2$, an $(\varepsilon, \ell-1, E)$-good set $A_0$ is $(\varepsilon, \ell, E)$-good if, 
for $1 \le j \le \ell$,
we have $A_1, \ldots, A_{j-1} \in \hat{M} \setminus M$ such that 
$A_1$ is $(\varepsilon, j-1 , E)$-good, $A_2$ is $(\varepsilon, j-2, E)$-good, $\ldots$, and $A_{j-1}$ is $(\varepsilon, 1, E)$-good, 
then any $\varepsilon$-partial relation which first varies $A_0$, then varies $A_1$, $\dots$, and finally varies $A_{j-1}$ always returns a truth value (no matter what the remaining arguments are from $M$). 

Note that if $A$ is $(\varepsilon, \ell, E)$-good and $1 \le \ell^* < \ell$ then $A$ is also $(\varepsilon, \ell^*, E)$-good. So in particular, if $A$ is $\varepsilon$-excellent then $A$ is $(\varepsilon, \ell^*, E)$-good for all $\ell^* \leq n$, where $n=\arity(E)$. This means that if $A_0, \dots, A_{n-1}$ are all $\varepsilon$-excellent then $\hat{E}^\mbar_\varepsilon(A_0, \dots, A_{n-1})$ must have a truth value. 
We can preserve goodness while weakening the tolerance $\varepsilon$, leading to the following straightforward but crucial observation. 

\begin{lemma}
	Let
$E \in \Lang$, and suppose $1 \le \ell^* \le \ell$ and $0 < \varepsilon \le \varepsilon^*$. If $A\in \hat{M}\setminus M$ is $(\varepsilon, \ell, E)$-good, then $A$ is $(\varepsilon^*, \ell^*, E)$-good. 
\end{lemma}

\cref{Proposition: Order of unraveling doesn't matter for edge relation on sets} 
tells us that when we have 
$(\varepsilon, \ell, E)$-good sets
$A_0, \dots, A_{\ell-1}$, if $A_0, \dots, A_{\ell-1}$ are the only arguments of $\hat{E}^\mbar_\varepsilon$ coming from $\hat{M} \setminus M$, then $\hat{E}^\mbar_\varepsilon$ has a truth value that is independent of the ordering of $\mbar$.  
As a consequence, we obtain a key result, \cref{an-excellent-corollary}, which says  
that the truth value of any $\varepsilon$-partial relation whose arguments are all $\varepsilon$-excellent does not
depend on the order
in which this value is calculated.

\begin{proposition}
\label{Proposition: Order of unraveling doesn't matter for edge relation on sets}
Let 
	$E \in \Lang$ have arity $n$, and suppose that $0 < \varepsilon < \frac{1}{4}$ and $1\le \ell \le n$.  Let $A_0, \dots, A_{\ell-1}\in\hat{M}\setminus M$ be $(\varepsilon, \ell, E)$-good sets, and let $A_\ell, \dots, A_{n-1} \in M$. For any two injective functions $\beta_0, \beta_1\colon \{0, \dots, \ell-1\}\to \{0, \dots, \ell-1\}$ and any permutation $\sigma$ of $\{0, \dots, n-1\}$,
\[
\hat{E}^{\sigma \circ \beta_0}_{\varepsilon}(A_{\sigma(0)}, \dots, A_{\sigma(n-1)}) =
\hat{E}^{\sigma \circ \beta_1}_{\varepsilon}(A_{\sigma(0)}, \dots, A_{\sigma(n-1)}) \in \{\top, \bot\}.
\]
\end{proposition}
\begin{proof}
Without loss of generality we may assume that $\sigma = \id$, as the proof of the general case is the same. 
	Our proof proceeds by induction on $\ell$.  \nl

\noindent \ul{\it Case $\ell = 1$:} \nl
	We have $\hat{E}^{\sigma \circ \beta_0}_{\varepsilon}(A_{\sigma(0)}, \dots, A_{\sigma(n-1)}) =
	  \hat{E}^{\sigma \circ \beta_1}_{\varepsilon}(A_{\sigma(0)}, \dots, A_{\sigma(n-1)})$
	because	$\beta_0 = \beta_1$, and these return a truth value by the definition of $(\varepsilon, 1, E)$-goodness. \nl\nl
\ul{\it Case $\ell> 1$:} \nl
As every permutation of $\{0, \dots, \ell-1\}$ is equal to a composition of transpositions, it suffices to prove the result when $\beta_0$ is a transposition of $\beta_1$. Therefore, we may assume without loss of generality that $\beta_0 = \<\ell-1, \dots, 2, 1, 0\>$ and $\beta_1 = \<\ell-1, \dots, 2, 0, 1\>$.

	Define 
\[ H^{1,0} \defas \hat{E}^{\beta_0}_{\varepsilon}(A_0, A_1, A_2, \dots, A_{n-1}) \]
and
\[ H^{0,1} \defas \hat{E}^{\beta_1}_{\varepsilon}(A_0, A_1, A_2, \dots, A_{n-1}). \]
Then our goal is to show 
	that $H^{1,0} = H^{0,1}$.
	First observe that they both have truth values because $A_0$ and $A_1$ are $(\varepsilon, \ell, E)$-good). We now show that they have the same truth value.

Suppose $H^{1,0} = \top$. Then there are at most
	\[
		(\varepsilon \cdot |A_0|) \cdot |A_1| + ((1- \varepsilon) |A_0|) \cdot (\varepsilon \cdot |A_1|)
		\]
	many pairs $(a, b) \in A_0 \times A_1$ such that $\hat{E}^{\<\ell-1, \dots, 2\>}_\varepsilon(a, b, A_2, \dots, A_{n-1}) = \bot$. 

Similarly, if $H^{0,1} = \bot$, then there are at most 
\[
(\varepsilon \cdot |A_1|) \cdot |A_0| + ((1- \varepsilon) \cdot |A_1|) \cdot (\varepsilon \cdot |A_0|)
\]
many pairs $(a, b) \in A_0 \times A_1$ such that $\hat{E}^{\<\ell-1, \dots, 2\>}_\varepsilon(a, b, A_2, \dots, A_{n-1}) = \top$. 

Hence if 
	\begin{align*}
		|A_0||A_1| &> (\varepsilon \cdot |A_0|) \cdot |A_1| 
		+ ((1- \varepsilon) |A_0|) \cdot (\varepsilon \cdot |A_1|) + \\
		& \hspace*{15pt} (\varepsilon \cdot |A_1|) \cdot |A_0| + ((1- \varepsilon) |A_1|) \cdot (\varepsilon \cdot |A_0|) \\
		& = 2(2\varepsilon- \varepsilon^2) |A_0||A_1|, 
	\end{align*}
then $H^{1,0} = \top$ and $H^{0,1} = \bot$ cannot both hold simultaneously. 

A similar calculation shows that if 
\[
|A_0||A_1| > 2(2\varepsilon- \varepsilon^2) |A_0||A_1| \\
\]
then $H^{1,0} = \bot$ and $H^{0,1} = \top$ cannot both hold simultaneously. 

Now, $\varepsilon < \frac{1}{4}$, and so 
$2(2\varepsilon- \varepsilon^2)  < 1$. Hence 
	$H^{1,0} = H^{0,1}$, and
	the result follows.
\end{proof}

From now on we will assume that $\varepsilon < \frac{1}{4}$. 

Let $\ell \ge 1$ and 
suppose $A_0, \dots, A_{n-1}\in \hat{M}$ are such that exactly $\ell$ are $(\varepsilon, \ell, E)$-good and exactly $n-\ell$ are in $M$, where $n = \arity(E)$.
(In particular, this occurs when each of $A_0, \dots, A_{n-1}$ is $\varepsilon$-excellent.)
Then by
\cref{Proposition: Order of unraveling doesn't matter for edge relation on sets}, 
$\hat{E}_\varepsilon^\mbar(A_0, \dots, A_{n-1})$ has a truth value that is independent of the $\ell$-tuple $\mbar$. In this case, we refer to $\hat{E}_\varepsilon^\mbar$ simply as $\hat{E}_{\varepsilon}$. 
This gives the following corollary.

\begin{corollary}
	\label{an-excellent-corollary}
	For any $\varepsilon$-excellent elements $A_0, \dots, A_{n-1} \in \hat{M}\setminus M$, and any $E \in \Lang$ of arity $n$, we have $\hat{E}_{\varepsilon}(A_0, \dots, A_{n-1}) \in \{\top, \bot\}$.
\end{corollary}

The following technical lemma tells us that, for a relation $E$ and appropriately good sets, at most a small fraction of the tuples consistent with those sets disagree 
with the partial relation $\hat{E}^\mbar_\varepsilon$ 
about the truth value of $E$.

\begin{lemma}
\label{Lemma: Number of mistakes in approximate truth}
	Let $E \in \Lang$ have arity $n$ and suppose that $0 < \varepsilon < \frac{1}{4}$ and $1\le \ell \le n$.
	Let $A_0, \dots, A_{\ell-1}\in\hat{M}\setminus M$ be such that 
	$A_i$ is $(\varepsilon, \ell-i, E)$-good
	for $0 \leq i < \ell$,
	and let $A_\ell, \dots, A_{n-1} \in M$.
	Let $\sigma$ be a permutation of $\{0, \dots, n-1\}$. 
Define	
	\begin{align*}
		Z \defas \{(a_0, \dots, a_{n-1}) \st ~~& a_i \in A_{\sigma^{-1}(i)}\mathrm{~when~}
		\sigma^{-1}(i) < \ell, \\
		\mathrm{~and~} &a_i = A_{\sigma^{-1}(i)}\mathrm{~when~}
		\sigma^{-1}(i) \ge \ell\}.
	\end{align*}
Then the following hold.
\begin{itemize}
	\item If $\hat{E}^{\<\sigma^{-1}(\ell-1),\dots,\sigma^{-1}(1), \sigma^{-1}(0)\>}_{\varepsilon}(A_{\sigma(0)}, A_{\sigma(1)}, \dots, A_{\sigma(n-1)}) = \top$ then 
\[
\bigl|\{(a_0, a_1, \dots, a_{n-1}) \in Z\st \M \models \neg E(a_0, a_1, \dots, a_{n-1})\}\bigr| \,\leq\,  \ell \cdot \varepsilon \cdot \prod_{0 \leq i < \ell} |A_i|.
\]

\item If $\hat{E}^{\<\sigma^{-1}(\ell-1),\dots,\sigma^{-1}(1), \sigma^{-1}(0)\>}_{\varepsilon}(A_{\sigma(0)}, A_{\sigma(1)}, \dots, A_{\sigma(n-1)}) = \bot$ then 
\[
\bigl|\{(a_0, a_1 \dots, a_{n-1}) \in Z\st \M \models E(a_0, a_1, \dots, a_{n-1})\}\bigr|
		\,\leq\,  \ell \cdot \varepsilon \cdot \prod_{0 \leq i < \ell} |A_i|.
\]
\end{itemize}
\end{lemma}

\begin{proof}
The proofs of the two bullet points are essentially identical so we will only prove the first. Further we can assume without loss of generality that $\sigma = \id$. 
	To simplify notation we will omit the superscript of the partial relation and refer to $\hat{E}^{\<\ell-1,\dots,0\>}_\varepsilon$ by $\hat{E}_\varepsilon$.

	Define the $\ell$-ary relation
	$F(x_0, \dots, x_{\ell-1}) \defas E(x_0, \dots, x_{\ell-1}, A_{\ell}, \dots, A_{n-1})$. 
	Note that as $A_i$ is $(\varepsilon, \ell-i, E)$-good, $A_i$ is also $(\varepsilon, \ell-i, F)$-good.

	At stage $m < \ell$,
	we recursively define
	an $\ell$-ary relation $F^m$ on $M$ and 
	for every $\c \in \prod_{i < m}A_i$,
	a unary relation $B_m^\c$ and 
	$\ell$-ary relation $C_m^\c$ on $M$,
	such that the following two inductive hypotheses hold. First,
\begin{itemize}
\item[($1_m$)] 
	whenever $\<a_0, \ldots, a_{\ell -1}\> \in \prod_{i< \ell}A_i$ 
		and $a_j \in B_j^{\<a_0, \dots, a_{j-1}\>}$ for some $j \leq m$ then 
		\[
			F^m(a_0, \dots, a_{\ell-1}) = \top,
			\]
\end{itemize}
		and second,
\begin{itemize}
\item[($2_m$)] 
	whenever $\<a_0, \ldots, a_{\ell-1}\> \in \prod_{i<\ell} A_i$ 
		and $a_j \not\in B_j^{\<a_0, \dots, a_{j-1}\>}$ for all $j \leq m$ then 
		\[
			\hat{F}_{\varepsilon}(a_0, \dots,a_{m}, A_{m+1}, \dots,  A_{\ell-1}) = \top.
			\] 
\end{itemize}
	Further, we will have $F \subseteq F^0 \subseteq F^1 \subseteq \cdots \subseteq F^{\ell-1}$.

\noindent \ul{Stage $0$:}\nl
	Let 
	\[
		B_0^\emptyset \defas \{c_0 \in A_0\st \hat{F}_{\varepsilon}(c_0, A_{1}, \dots, A_{\ell- 1})  = \bot\},
		\]
	and let
	\[
		C_0^\emptyset \defas 
		 B_0^\emptyset \times \prod_{0 < i < \ell} A_i.
	\]

	Then define $F^0 \defas F \cup C_0^\emptyset$ (where we consider the relation $F$ as a subset of $M^\ell$). 
	Condition $(1_0)$ holds because $C_0^\emptyset \subseteq F^0$.

	Because $A_i$ is $(\ell-i)$-good for $1 \le i < \ell$,  whenever $a_0 \not \in B_0^\emptyset$ we have 
	$\hat{F}_{\varepsilon}(a_0, A_{1}, \dots,  A_{\ell-1}) = \top$,
	and so condition $(2_0)$ holds.

\noindent \ul{Stage $k$, where $1\le  k < \ell$:}\nl
	Suppose that for $j<k$
	the relations $F^j$, and 
	for $\d \in \prod_{i < j}A_j$
	the relation $B_j^\d$, satisfy conditions $(1_{j})$ and $(2_{j})$.
	We now show how to appropriately define $F^{k}$, $B_{k}^\c$, and $C_{k}^\c$ for parameters $\c\in \prod_{i < k}A_i$ of length $k$.

	Suppose 
	$\<c_0, \dots, c_{k-1}\> \in \prod_{i < k}A_i$.
	If 
	$c_j \in B_j^{\<c_0, \dots, c_{j-1}\>}$ for some $j < k$, 
	then let $B_{k}^{\<c_0, \dots, c_{k-1}\>} \defas \emptyset$ and $C_{k}^{\<c_0, \dots, c_{k-1}\>} = \emptyset$. 
	Otherwise, 
	$c_j \not\in B_j^{\<c_0, \dots, c_{j-1}\>}$ for all $j < k$,
	in which case we
	define
	\[
		B_{k}^{\<c_0, \dots, c_{k-1}\>} \defas \bigl\{c_{k}\in A_{k} \st \hat{F}_\varepsilon(c_0, \dots, c_{k-1}, c_{k}, A_{k+1}, \dots, A_{\ell-1}) = \bot\bigr\}.
	\]
and 
\begin{eqnarray*}
	C_{k}^{\<c_0, \dots, c_{k-1}\>} \defas 
	\bigl\{\<c_0, \dots, c_{k-1}\>\bigr\}
	\times 
		 B_{k}^{\<c_0, \dots, c_{k-1}\>} \times \prod_{k+1 \le i < \ell} A_i.
\end{eqnarray*}
	Finally, define 
	\[
		F^{k} \defas F^{k-1} \cup \bigcup_{\<c_0, \dots,c_{k-1}\> \in \prod_{i < k} A_i} \bigl \{C_{k}^{\<c_0, \dots,c_{k-1}\>}\st
			\hat{F}_{\varepsilon}(c_0, \dots,c_{k-1}, A_{k}, \dots,  A_{\ell-1}) = \top \bigr\}.
			\]

			We now show that condition $(1_{k})$ holds.
Let $\<a_0, \dots, a_{\ell-1}\>\in \prod_{i < \ell} A_i$, and suppose
$a_j \in B_j^{\<a_0, \dots, a_{j-1}\>}$ for some $j \le k$. 
If $a_j \in B_j^{\<a_0, \dots, a_{j-1}\>}$ for some $j < k$ then 
		by condition $(1_{j})$ we have
			$F^j(a_0, \dots, a_{\ell-1}) = \top$. Hence
			$F^k(a_0, \dots, a_{\ell-1}) = \top$ also as $F^j \subseteq F^k$.
Otherwise, we have
(i)		$a_j \not \in B_j^{\<a_0, \dots, a_{j-1}\>}$ for all $j < k$ 
		and (ii) $a_k \in B_k^{\<a_0, \dots, a_{k-1}\>}$.
		By (ii), we have $\<a_0, \dots,a_{\ell-1}\> \in  C_{k}^{\<a_0, \dots,a_{k-1}\>}$.
		By (i) and condition $(2_{k-1})$ we have
		\[
			\hat{F}_{\varepsilon}(a_0, \dots,a_{k-1}, A_{k}, \dots,  A_{\ell-1}) = \top,
			\] 
			and so 
		$C_{k}^{\<a_0, \dots,a_{k-1}\>}\subseteq F^k$.
			Therefore $F^k(a_0, \dots, a_{\ell-1}) = \top$.

			Towards showing condition $(2_k)$, again let $\<a_0, \dots, a_{\ell-1}\>\in \prod_{i < \ell} A_i$
			and suppose that $a_j \not \in B_j^{\<a_0, \dots, a_{j-1}\>}$ for all $j \le k$.
			By the definition of $B_k^{\<a_0, \dots, a_{k-1}\>}$, and because each $A_i$ is $(\ell-i)$-good for $k+1 \le i < \ell$, 
			we have
			\[
				\hat{F}_{\varepsilon}(a_0, \dots,a_{k}, A_{k+1}, \dots,  A_{\ell-1}) = \top.
				\]

\vspace*{10pt}

To conclude the proof, consider the relation $F^{\ell-1}$. By our assumption in the first bullet point,
we have
$F^{\ell-1}(a_0, \dots, a_{\ell-1}) = \top$ for all $(a_0, \dots, a_{\ell-1}) \in \prod_{j < \ell} A_j$. In other words, $\prod_{j < \ell} A_j \subseteq F^{\ell-1}$.
Because the last $n-\ell$ terms of a tuple in $Z$ are fixed, we have
\[
\bigl|\bigl\{(a_0, \dots, a_{n-1}) \in Z\st \M \models \neg E(a_0, \dots, a_{n-1})\bigr\}\bigr| \, \leq \, \bigl|F^{\ell-1}\setminus F\bigr|.
\]
By the definitions of $F_j$ for $j<\ell$, we have
\[
	F^{\ell-1} \setminus F \subseteq \bigcup_{j < \ell} \bigcup \bigl\{C_{j}^{\a}\st \a \in \prod_{i < j}A_i\bigr\}.
	\]

As each $A_j$ is $(\varepsilon, \ell-j, E)$-good, for 
each $\a \in \prod_{i < j}A_i$ we have
\[
|C_{j}^{\a}| \leq \varepsilon \cdot \prod_{j \leq i < \ell} |A_i|,
\]
and so 
\[
\Bigl|\bigcup \bigl\{C_{j}^{\a}\st \a \in \prod_{i < j}A_i\bigr\}\Bigr| \leq  \varepsilon \cdot \prod_{i < \ell} |A_i|. 
\]
But then $|F^{\ell-1} \setminus F| \leq \ell \cdot \varepsilon \cdot \prod_{i < \ell} |A_i|$, as desired.  
\end{proof}

As a consequence of \cref{Lemma: Number of mistakes in approximate truth}, 
we show in \cref{Changes in edges due to excellence needed to get indivisible} that
given a partition of $\M$
into $\varepsilon$-excellent parts, we can assign 
a consensus truth value
to any relation $E$ and $\arity(E)$-tuple of parts of the partition.
This produces a partition that is almost indivisible (with respect to $\M$) in the following sense.

\begin{proposition}
\label{Changes in edges due to excellence needed to get indivisible}
	Let $P$ be an equitable partition of $\M$
	such that each part of $P$ is $\varepsilon$-excellent, and let $E \in \Lang$ of arity $n$.
	Then there is an $n$-ary relation $E^*$ on $M$ such that
	for all tuples $\<p_i\>_{i < n}$ from $P$,
		\[
			\textstyle
			\left|(E \Delta E^*) \cap \prod_{i < n} p_i\right| \leq n \cdot \varepsilon \cdot \prod_{i < n}|p_i|,
			\]
		and	$P$ is an indivisible partition of the structure $(M, E^*)$ with underlying set $M$ and the relation $E^*$. 
\end{proposition}

\begin{proof}
	If $p_0, \dots, p_{n-1} \in P$,
	then $\hat{E}_{\varepsilon}(p_0, \dots, p_{n-1})$ has a truth value,
	because each 
	part of $P$ is $\varepsilon$-excellent; further, if $\hat{E}_{\varepsilon}(p_0, \dots, p_{n-1}) = \top$ then
\[
	\textstyle
\left|\prod_{i < n}p_i \setminus 
	E
	\right| \leq n \cdot \varepsilon \cdot \prod_{i < n}|p_i|
\]
	holds by \cref{Lemma: Number of mistakes in approximate truth}, and analogously when $\hat{E}_{\varepsilon}(p_0, \dots, p_{n-1}) = \bot$.

Now let $E^*\subseteq M^n$ be such that for any $p_0, \dots, p_{n-1} \in P$,
	if $\hat{E}_{\varepsilon}(p_0, \dots, p_{n-1}) = \top$ then $E^* \cap \prod_{i < n} p_i = \prod_{i < n}p_i$, and if $\hat{E}_{\varepsilon}(p_0, \dots, p_{n-1}) = \bot$ then $E^* \cap \prod_{i < n} p_i = \emptyset$. It is then clear that $(M, E^*)$ is indivisible. 
\end{proof} 

Applying 
\cref{Changes in edges due to excellence needed to get indivisible} to each relation $E\in\Lang$, in aggregate we obtain an
$\Lang$-structure $(M, E^*)_{E\in\Lang}$
that is indivisible.
Because $(M, E^*)_{E\in\Lang}$ is obtained from $\M$ by a small number of modifications of each $E\in\Lang$ to obtain the corresponding $E^*$, we may think of $\M$ itself as almost indivisible.

\section{Obtaining excellent sets}
In this section, we will show how to use the fact that a finite $\Lang$-structure $\M$ with underlying set $M$ has the non-\brap-branching property to get large excellent sets. Specifically, we start with a set $A$ and try and build a binary-branching tree of subsets of $A$, where the set at a child node has size at least $\varepsilon$ times the size of the set at the parent node, and where the sets at any two children disagree on some ``question'' that excellent sets ``decide''. If this process of building a tree terminates, then there must be some set which we could not divide into two pieces each of size an $\varepsilon$ fraction of the set, each of which gives a different answer to a question that $\varepsilon$-sets can answer. Hence we will deduce that such a set must itself be $\varepsilon$-excellent. We will then show that such a tree must have a height bounded by a term definable from $\brap$, which will give us a bound on how large (as a fraction of our original set) an $\varepsilon$-excellent set we can find. 

In addition, when such a tree branches we will further require the subsets at the children nodes to be not merely ``sufficiently large'', but also one of a given predetermined set of sizes. In this way we will ensure that the sizes of all $\varepsilon$-excellent sets we create have a large greatest common divisor. This will be important when, in \cref{Section: Equitable Partitions of Excellent Sets}, we wish to divide our partition of $\varepsilon$-excellent sets
into an equitable partition of $\varepsilon$-excellent sets. 

\begin{definition}
	A \defn{rock} is a tuple $\<A, Q, \ell, (B^0, \ldots, B^{\ell - 1}, B^{\ell + 1}, \ldots, B^{\arity(Q)-1}), \beta\>$, where
	\begin{itemize}
\item $A\in \Powerset(M)\setminus \emptyset$,

\item $Q$ is a relation symbol in $\Lang$,

\item $\ell \in \Nats$ such that $\ell < \arity(Q)$,

\item each $B^t \in \Powerset(M)\setminus \emptyset$, and

\item $\beta\colon \{1, \dots, \arity(Q)-1\} \to \{0, \dots, \arity(Q)-1\}\setminus \{\ell_i\}$ is an injection (and hence a bijection).
\end{itemize}
	We say that such a rock \defn{covers} the set $A$.
\end{definition}
\begin{definition}
	Let $k \in \Nats$
and $\varepsilon > 0$. 
	A finite tuple $\<m_j\>_{j \leq k}$ of positive integers is a \defn{staircase} if $\frac{m_{j+1}}{m_{j}}\leq \varepsilon$ for all $j < k$.
\end{definition}

\begin{definition}
\label{Definition: Mesa}
	Let $k \in \Nats$
and $\varepsilon > 0$, and suppose 
	$\mbf{m}\defas \<m_j\>_{j \leq k}$ is a staircase.
	Define an \defn{$(\varepsilon, \mbf{m})$-mesa} of height $k$ to consist of a tree of rocks
	\[
		\bigl\<\bigl(A_i, Q_i, \ell_i, (B^0_i, \ldots, B^{\ell - 1}_i, B^{\ell + 1}_i, \ldots, B^{\arity(Q)-1}_i), \beta_i\bigr)\bigr\>_{i\in \{0, 1\}^{<k}}
		\]
		along with a collection of sets (called \defn{pre-caps}) $\<A_i\>_{i \in \{0, 1\}^k}$ indexed by the children of the leaves,
that satisfy,
	for each $i \in \{0, 1\}^{<k}$,
\begin{itemize}
\item 
	$B_i^j$ is $(\varepsilon, \arity(Q_i) - \beta_i^{-1}(j), Q_i)$-good
	for each $j \in \{0, \dots, \arity(Q_i)-1\} \setminus \{\ell_i\}$.

		\vspace*{7pt}
\item  
	$|A_{i\^s}| \in\mbf{m}$ and
		$|A_{i\^s}| \geq |\varepsilon| \cdot |A_i|$
	for each $s \in \{0,1\}$.

		\vspace*{7pt}
\item 
		$\bigl(\hat{Q_i}\bigr)^{\<\beta_i(\arity(Q_i)-1), \ldots, \beta_i(1)\>}_{\varepsilon}(B_i^0, \dots, B_i^{\ell_i - 1}, a, B_i^{\ell_i + 1}, \dots, B_i^{\arity(Q_i)-1}) = \bot$
	for all $a \in A_{i\^0}$.
		\vspace*{7pt}
\item 
	$\bigl(\hat{Q_i}\bigr)^{\<\beta_i(\arity(Q_i)-1), \ldots, \beta_i(1)\>}_{\varepsilon}(B_i^0, \dots, B_i^{\ell_i - 1}, a, B_i^{\ell_i + 1}, \dots, B_i^{\arity(Q_i)-1}) = \top$
	for all $a \in A_{i\^1}$.
\end{itemize}

	Consider an $(\varepsilon, \mbf{m})$-mesa as above, suppose $m_{k+1}$ is such that $\frac{m_{k+1}}{m_k} \leq \varepsilon$, and let
	$A_p$ be a pre-cap such that 
	$\varepsilon \cdot |A_p| \le  m_{k+1}$. 
	Then $A_p$ is an \defn{$m_{k+1}$-cap} if
	there is no rock 
	$\<A_p, Q, \ell, (B^0, \ldots, B^{\ell - 1}, B^{\ell + 1}, \ldots, B^{\arity(Q)-1}), \beta\>$
	covering it 
	such that 
\[
	m_{k+1} \leq
\{a \in A_p \st \hat{Q}^{\<\beta(\arity(Q)-1), \ldots, \beta(1)\>}_{\varepsilon}(B^0, \dots, B^{\ell - 1}, a, B^{\ell + 1}, \dots B^{\arity(Q)-1}) = \bot\} 
\]
and 
\[
	m_{k+1} \leq
\{a \in A_p \st \hat{Q}^{\<\beta(\arity(Q)-1), \ldots, \beta(1)\>}_{\varepsilon}(B^0, \dots, B^{\ell - 1}, a, B^{\ell + 1}, \dots B^{\arity(Q)-1}) = \top\}.
\]
A \defn{cap} of an $(\varepsilon, \mbf{m})$-mesa is an $m_{k+1}$-cap of the mesa for some $m_{k+1} \le \varepsilon m_k$.

An $(\varepsilon, \mbf{m})$-mesa has \defn{constant location $\ell$} if $\ell_i = \ell$ for all $i \in \{0, 1\}^{<k}$, and has \defn{constant relation $Q$} if $Q_i = Q$ for all $i \in \{0, 1\}^{<k}$.  

Let $Y$ be an
	$(\varepsilon, \mbf{m})$-mesa, and 
	suppose $\mbf{m'}$ has $\mbf{m}$ as an initial segment.
Then 
an $(\varepsilon,\mbf{m'})$-mesa 
	$Z$ is an \defn{extension} of $Y$ if (i) $Z$ extends $Y$ (as a tree of rocks), and 
	(ii) $Z$ at the level after the height of $Y$ contains,
	for each pre-cap of $Y$, 
	a rock that covers that pre-cap.

Suppose $m_{k+1}$ is such that $\frac{m_{k+1}}{m_k} \leq \varepsilon$. An $(\varepsilon, \mbf{m})$-mesa is \defn{$m_{k+1}$-maximal} if it has no extensions which are $(\varepsilon, \mbf{m}\^m_{k+1})$-mesas.
\end{definition}

Note that if $C$ is the cap of a mesa, then every rock covering $C$ determines the truth value of its relation symbol (with its arguments and its ordering), in the sense that there is only one truth value that a large fraction of $C$ agrees with.

\begin{lemma}
\label{Lemma: Bounds on size of leaves of mesas}
Let $Y$ be 
	an $(\varepsilon, \mbf{m})$-mesa
	with notation as in \cref{Definition: Mesa}.
Let $m_{k+1} \leq \varepsilon m_k$, and suppose that $Y$ is $m_{k+1}$-maximal.
\begin{itemize}
	\item[(a)] Let $p \in \{0, 1\}^k$. If the pre-cap $A_p$ is an $m_{k+1}$-cap of $Y$, then $A_p$ is $\varepsilon$-excellent. 

	\item[(b)] There is a (not necessarily unique) $m_{k+1}$-cap of $Y$.
\end{itemize}
\end{lemma}
\begin{proof}
	(a) This follows immediately from the definition of $m_{k+1}$-cap and the fact that $\frac{m_{k+1}}{m_k} \leq \varepsilon$. 

	(b) If there is no $m_{k+1}$-cap for any $p \in \{0, 1\}^k$, then by the definition of an $(\varepsilon, \mbf{m})$-mesa we can find an extension of $Y$ to an $(\varepsilon, \mbf{m}\^m_{k+1})$-mesa, contradicting the assumption that $Y$ was $m_{k+1}$-maximal.
\end{proof}

In fact, an $(\varepsilon, \mbf{m})$-mesa is $m_{k+1}$-maximal if and only if it has some $m_{k+1}$-cap.

We will eventually want to obtain a bound on the height of an $(\varepsilon, \mbf{m})$-mesa based on the underlying $\Lang$-structure $\M$ having the non-$\brap$-branching property. To do this, we will need an $(\varepsilon, \mbf{m})$-mesa with constant relation and constant location. 

We first define what it means for a mesa to be a substructure of another.

\begin{definition}
	Let $k, k*\in \Nats$, let $\varepsilon > 0$, and suppose 
	$\mbf{m}\defas \<m_j\>_{j \leq k}$  and
	$\mbf{m^*}\defas \<m^*_j\>_{j \leq k^*}$  
	are staircases.
Let $Y$ be an $(\varepsilon, \mbf{m})$-mesa
and $Y^*$ an $(\varepsilon, \mbf{m^*})$-mesa.

	Then $Y^*$ is a \defn{substructure} of $Y$ if there are injective maps $\alpha\colon \{0, 1\}^{\leq k^*} \to \{0, 1\}^{\leq k}$ and $\gamma\colon \{0, \dots, k^*-1\} \to \{0, \dots, k-1\}$ such that, for all $i, i' \in  \{0, 1\}^{\leq k^*}$,
\begin{itemize}
\item $m^*_h = m_{\gamma(h)}$  for all $h \leq k^*$,

\item 
	$\len\bigl(\alpha(i)\bigr) = \gamma\bigl(\len(i)\bigr)$,

\item 
	if $i$ is an initial segment of $i'$ then $\alpha(i)$ is an initial segment of $\alpha(i')$,

\item if $\len(i) < k^*$, then 
		the rock of $Y$ at node $\alpha(i)$ equals
	the rock of $Y^*$ at node $i$,
\item  if $\len(i) = k^*$ and $\gamma(k^*) = k$, then 
		the pre-cap of $Y$ at node $\alpha(i)$ equals
	the pre-cap of $Y^*$ at node $i$, and
\item  if $\len(i) = k^*$ and $\gamma(k^*) < k$, then 
	the rock of $Y$ at node $\alpha(i)$ covers the
	the pre-cap of $Y^*$ at node $i$.
\end{itemize}

\end{definition}

We will soon show the key fact that for every $k^*\in \Nats$ there is some $k \ge k^*$, depending only on $k^*$, such that
every $(\varepsilon, \mbf{m})$-mesa of height at least $k$ has some substructure that is a $(\varepsilon, \mbf{m^*})$-mesa with constant location. We will use the following Ramsey-theoretic result about colored trees.

\begin{lemma}[{\cite[Theorem~2~(i)]{Ramsey}}]
\label{Lemma: Ramseys theorem for binary branching trees}
Let $p, q\ge 2$. Suppose $T$ is a binary branching tree of height at least $H > 5 \cdot q \cdot p \cdot \log p$ along with a map $\iota$ from the nodes of the tree to $\{0, \dots, q-1\}$. Then there is a binary branching tree $T^*$ and an injection $\alpha\colon T^* \to T$ such that 
\begin{itemize}
\item $T^*$ has height $p$,

\item $\alpha$ preserves the partial ordering of nodes in the tree, and preserves when two nodes are on the same level, and

\item $\iota \circ \alpha \colon T^* \to \{0, \dots, q-1\}$ is constant. 
\end{itemize}
\end{lemma}

\begin{lemma}
\label{Lemma: Mesa has constant location substructure}
Suppose $n_\Lang, k^* \ge 2$,
and suppose $Y$ is an $(\varepsilon, \mbf{m})$-mesa of height $k > 5 \cdot n_\Lang \cdot k^* \cdot \log k^*$. Then there is 
	some staircase $\mbf{m^*}$ of length $k^*$ and some
	substructure 
	$Y^*$ of $Y$ 
	that is an $(\varepsilon, \mbf{m^*})$-mesa 
	which has constant location and constant relation. 
\end{lemma}
\begin{proof}
This follows immediately from \cref{Lemma: Ramseys theorem for binary branching trees}.
\end{proof}

Our next step is to show how to get from a mesa having constant location and constant relation to a witness to the $k$-branching property.

For $E \in \Lang$ and $0 \leq \ell \leq \arity(E)-1$, write
\[
\mathcal{E}^{\ell}(x_\ell, x_0, \dots, x_{\ell - 1}, x_{\ell+1}, \dots, x_{\arity(E)-1}) \defas E(x_0, \dots, x_{\arity(E)-1}),
\]
so that we may easily isolate $x_\ell$ from the other variables when talking about stability.

\begin{lemma}
\label{Lemma: Yellow implies witness to branch}
Suppose there is an $(\varepsilon, \mbf{m})$-mesa $Y$ of height $k$ with constant location $\ell$ and constant relation $Q$, and suppose $2^k \cdot (\arity(E)-1) \cdot \varepsilon < 1$. Then $(M, \mathcal{E}^{\ell}(x_\ell; x_0, \dots, x_{\ell - 1},$ $x_{\ell+1}, \dots, x_{n-1}))$ has the $k$-branching property.
\end{lemma}
\begin{proof}
We use the notation for the components of $Y$ as in \cref{Definition: Mesa}. Without loss of generality, we may assume $\ell = 0$. For each $\eta \in \{0, 1\}^k$ let $a_{\eta} \in A_\eta$. Now for each $\eta \in \{0, 1\}^k$ and each $\nu \in \{0, 1\}^{< k}$ define 
	\begin{eqnarray*}
		U_{\nu, \eta}\defas \Bigl\{(b_1, \dots, b_{n-1}) \in \prod_{1 \leq j < n} B_\nu^j \st \hspace*{200pt}~\\
		\hat{E}^{\<\beta_\nu(n-1), \ldots, \beta_\nu(1)\>}_{\varepsilon}(a_{\eta}, b_1, \dots, b_{n-1}) \neq \hat{E}_{\varepsilon}^{\<\beta_\nu(n-1), \ldots, \beta_\nu(1)\>}(a_{\eta}, B_\nu^1, \dots, B_{\nu}^{n-1}) 
		\Bigr\}.
	\end{eqnarray*}
Now by \cref{Lemma: Number of mistakes in approximate truth}, 
	we have $|U_{\nu, \eta}| < (n-1) \cdot \varepsilon \cdot \prod_{1 \leq j < n} |B_\nu^j|$
	for every $\eta \in \{0, 1\}^k$ and $\nu \in \{0, 1\}^{< k}$.
	Hence 
	\[
		\Bigl|\bigcup_{\nu \preceq \eta} U_{\nu, \eta}\Bigr| < 2^k \cdot (n-1) \cdot \varepsilon \cdot \prod_{1 \leq j < n} |B_\nu^j| 
		\]
	for every $\nu \in \{0, 1\}^{<k}$. 

	But we assumed $2^k \cdot (n-1) \cdot \varepsilon < 1$, and so for any $\nu$ we can find some $\mbf{b}_\nu \defas (b_\nu^1, \dots, b_\nu^{n-1}) \in \prod_{1 \leq j < n-1}B_{\nu}^j \setminus \bigcup_{\nu \preceq \eta} U_{\nu, \eta}$. 

	But then by construction, $\<\mbf{b}_{\nu}\>_{\nu \in \{0, 1\}^{< k}}$ and $\<a_\eta\>_{\eta \in \{0, 1\}^k}$ witness that 
	\linebreak
	$\mathcal{E}^0(x_0;x_1, \dots,  x_{n-1})$ has the $k$-branching property. 
\end{proof}

Putting all of these together we get the following crucial proposition.

\begin{proposition}
\label{Proposition: Get supset which is epsilon-excellent}
	Let $\M$ be a finite $\Lang$-structure with underlying set $M$.
	Suppose that $\M$ does not have the $\brap$-branching property and
	that $0 < \varepsilon < 2^{-\brap} \cdot n_\Lang^{-1}$.
	Let
	$g = \lceil 5 \cdot n_\Lang \cdot \brap \cdot \log \brap \rceil$. 
	Further suppose that $\mbf{m} \defas \<m_i\>_{i \leq g}$ is a staircase, and 
	that $A\subseteq M$ is such that $|A| \geq m_0$.
	Then $A$
	contains an $\varepsilon$-excellent
	subset $A'$ of size $m_i$ 
	for some $i \leq g$. 
\end{proposition}
\begin{proof}
By \cref{Lemma: Branching property from order property}, for any $E \in \Lang$ and $\ell < \arity(E)$ the structure 
	\[(M, \mathcal{E}^{\ell}(x_\ell; x_0, \dots, x_{\ell - 1}, x_{\ell+1}, \dots, x_{\arity(E)-1}))\]
	has the non-$\brap$-branching property. 
	By our assumption on $\varepsilon$, we may apply \cref{Lemma: Yellow implies witness to branch}, and so any $(\varepsilon, \mbf{m})$-mesa of constant location and constant relation $E$ can have height at most $\brap$. But then by \cref{Lemma: Mesa has constant location substructure},
	the height of any $(\varepsilon, \mbf{m})$-mesa is at most $5 \cdot n \cdot \brap \cdot \log \brap $. 

	In particular there must be some $j \leq g$ and $(\varepsilon, \<m_i\>_{i \leq g})$-mesa which is $m_{j+1}$-maximal. But then by \cref{Lemma: Bounds on size of leaves of mesas} this mesa must have a cap, which has size $m_j$ for some $j \le g$. Further, by 
\cref{Lemma: Bounds on size of leaves of mesas} this cap
	is $\varepsilon$-excellent.
\end{proof}

Having developed a method to find a large $\varepsilon$-excellent subset of any sufficiently large subset of $M$, we now aim to find a partition of $\M$ 
such that (1) all but one part is $\varepsilon$-excellent and (2) for any two parts, the size of one divides the size of the other, along with a bound on the size of the non-$\varepsilon$-excellent part.

\begin{proposition}
\label{Proposition: Partition of excellent sets}
	Let $\M$ be a finite $\Lang$-structure with underlying set $M$.
	Suppose $0 < \varepsilon < 2^{-\brap} \cdot n_\Lang^{-1}$, and that
\begin{itemize}
\item $\M$ does not have the $\brap$-branching property,

\item $n = |\Lang| \cdot q_\Lang$,

\item $g = \lceil 5 \cdot n_\Lang \cdot \brap \cdot \log \brap \rceil$,

\item $r = \lfloor\frac{1}{\varepsilon}\rfloor$, and

\item $\mbf{m} \defas \<m_i\>_{i \leq g}$ is a staircase such that
\begin{itemize}
    \item 
		$\frac{m_{i}}{m_{i+1}} = r$ 
		for all $0 \leq i < g$,  and

    \item $|M| \geq m_0$.
\end{itemize}
\end{itemize}
Then there is a subset $M^* \subseteq M$ and a partition $P$ of $M^*$ such that 
\begin{itemize}
\item $|M\setminus M^*| <  m_0$,

\item each part of $P$ is $\varepsilon$-excellent, and

\item $|p| \in \mbf{m}$  for all $p \in P$.

\end{itemize}

\end{proposition}
\begin{proof}
	We define the partition by induction. For the base case, let $M_0 \defas M$ and let $P_0$ be an $\varepsilon$-excellent subset of $M_0$ with $|P_0|\in \mbf{m}$, as guaranteed by \cref{Proposition: Get supset which is epsilon-excellent}.
	
	For the inductive step, suppose we that have already defined $M_n$ and $\<P_j\>_{j \le n}$, where each $P_j$ is $\varepsilon$-excellent and whose size is in $\mbf{m}$.
	Let $M_{n+1} \defas M_n \setminus P_n$. 
	
	If $|M_{n+1}| < m_0$ then let $M^* \defas M \setminus M_{n+1}$ and let $P \defas \{P_i\}_{i \leq n}$; then $M^*$ and $P$ have the desired properties.

	Otherwise let $P_{n+1}$ be an $\varepsilon$-excellent subset of $M_{n+1}$ with $|P_{n+1}| \in \mbf{m}$, as guaranteed by \cref{Proposition: Get supset which is epsilon-excellent}, and proceed to the next step of the induction.
\end{proof}

\section{Equitable partitions of excellent sets}
\label{Section: Equitable Partitions of Excellent Sets}

We have just seen, 
in \cref{Proposition: Partition of excellent sets},
that a large subset of a 
sufficiently large structure $\M$ may be partitioned into $\varepsilon$-excellent sets.
In this section, we show, in \cref{Dividing epsilon-excellent classes into uniform excellent classes},
how to refine this into an
equitable partition of $\M$ into $(\varepsilon + \zeta)$-excellent
sets, for some $\zeta > 0$.

Then,
in the main results of this section, 
\cref{Proposition: Equitable excellent partition}
and
\cref{Theorem: Stable Regularity for Relational Structures with choice of partition size}, 
we show how to
uniformly distribute the elements of our structure not in this large subset, obtaining an equitable partition of the entire structure which witnesses
that it is close in edit distance to an equitable blow-up.

Throughout this section, let $\M$ be a finite $\Lang$-structure with underlying set $M$.

Our first lemma immediately implies that
if a set agrees with an $\varepsilon$-excellent set on the truth values of all edge relations in $\Lang$ with respect to all parameters that are elements of $M$, then the set itself must be $\varepsilon$-excellent.

\begin{lemma}
\label{(epsilon k)-good by proxy}
Let $E \in \Lang$ and let $n$ be the arity of $E$.  Suppose that $A$ is $(\varepsilon, k, E)$-good and that $A'$ is such that for all elements $b_1, \dots, b_{n-1} \in M$ and every permutation $\sigma$ of $n$,
\[
\hat{E}_{\varepsilon}(x_{\sigma(0)}, \dots, x_{\sigma(n-1)}) = \hat{E}_{\varepsilon}(y_{\sigma(0)}, \dots, y_{\sigma(n-1)}) 
\]
where $x_0 = A$ and $y_0 = A'$, and
	$x_i = y_i = b_i$ whenever $1 \leq i < n$.
	Then $A'$ is $(\varepsilon, k, E)$-good. 
\end{lemma}
\begin{proof}
We will prove the following statement $(*_k)$ by induction on $k$:

\noindent 	$(*_k)$: For all $b_{k-1}, \dots, b_n \in M$
and	permutations $\sigma$ of $n$, 
if 	
$B_i$ is $(\varepsilon, k-i, E)$-good
	for all $1 \le i < k$, then 
	\[
		\hat{E}^{\sigma^+}_{\varepsilon}(x_{\sigma(0)}, \dots, x_{\sigma(n-1)}) = \hat{E}^{\sigma^+}_{\varepsilon}(y_{\sigma(0)}, \dots, y_{\sigma(n-1)}) 
\]
where $x_0 = A$ and $y_0 = A'$, 
	where $\sigma^+\defas  \sigma|_{\{0, \ldots, k-1\}}$,
	and 
	$x_i = y_i = B_i$ 
	whenever $1 \leq i < k$, 
	and $x_i = y_i = b_i$ whenever $k \leq i < n$.

	\vspace*{10pt}	

\noindent \ul{\it Case $k = 1$:} \nl
This is immediate by our assumption.  \nl\nl
\ul{\it Case $k > 1$:} \nl
By the inductive assumption, $A'$ is $(\varepsilon, k-1, E)$-good.
	We must show that it is $(\varepsilon, k, E)$-good.

Now suppose $A_1, \dots, A_{k-1}\subseteq M$ and $a_k, \dots, a_{n-1} \in M$, where $A_i$ is $(\varepsilon, k -i, E)$-good whenever $1 \le i < k$.
	Without loss of generality, it suffices to show that
\[
	\hat{E}^\id_{\varepsilon}(A, A_1, \dots, A_{k-1}, a_k, \dots, a_{n-1}) = 
	\hat{E}^\id_{\varepsilon}(A', A_1, \dots, A_{k-1}, a_k, \dots, a_{n-1}),
\]
where $\id$ is the identity map on $\{0, \ldots, k-1\}$.
But we know that
	\[
		\hat{E}^\id_{\varepsilon}(A, A_1, \dots, A_{k-1}, a_k, \dots, a_{n-1}) \in \{\top, \bot\}.\]
	Suppose that 
	$\hat{E}^\id_{\varepsilon}(A, A_1, \dots, A_{k-1}, a_k, \dots, a_{n-1}) = \top$. Then 
\[
	\frac{|\{a \in A_{k-1}\st \hat{E}^\id_{\varepsilon}(A, A_1, A_2, \dots, A_{k-2}, a, a_k, \dots, a_{n-1}) = \top\}|}{|A_{k-1}|} \geq 1-\varepsilon.
\]
But then by the inductive hypothesis we also have
\[
	\frac{|\{a \in A_{k-1}\st \hat{E}^\id_{\varepsilon}(A', A_1, A_2, \dots, A_{k-2}, a, a_k, \dots, a_{n-1}) = \top\}|}{|A_{k-1}|} \geq 1-\varepsilon.
\]
Hence $\hat{E}^\id_{\varepsilon}(A', A_1, \dots, A_{k-1}, a_k, \dots, a_{n-1}) = \top$. 

The case when $\hat{E}^\id_{\varepsilon}(A, A_1, \dots, A_{k-1}, a_k, \dots, a_{n-1}) = \bot$ is identical. 
\end{proof}

Now we want to show that if our $\varepsilon$-excellent set is sufficiently large then a uniformly random equitable partition will be $(\varepsilon+\zeta)$-excellent with high probability, for some $\zeta$.

\begin{lemma}
\label{Lemma: Polynomial bound on number of bad sets}
If $\varphi(\x;\y)$ has the non-$\ordp$-order property in a structure $\M$ then for any finite $A \subseteq \M$ with $|A| > 2$, 
\[
|\{\{\a\in A\st \varphi(\a, \b)\} \st \b \in \M\}| \leq |A|^\ordp.
\]
\end{lemma}
\begin{proof}
	This is immediate from \cite[Theorem~II.4.10(4)]{Classification-theory}.
\end{proof}

The following result provides an upper bound on the probability that the fraction of elements satisfying property $S$ will be more than the expected value by an additive constant $t$. 

\begin{proposition}[{\cite{hypergeometric}}]
\label{Proposition: Bound on the cap of the hypergeometric}
Suppose we have $N$ elements of which $K$ have a property $S$. Let $H(n, N, K)$ be the random variable which selects without replacement $s$ elements and returns the number which have property $S$. Then for any $t > 0$ we have 
\[
\Pr\left[\frac{H(s, N, K)}{s} \geq \frac{K}{N} + t\right] \leq e^{-2t^2 s}.
\]
\end{proposition} 

For our purposes we will have an $\varepsilon$-excellent set $A$ and we will want to sample a random partition $P$ of $A$. We will then want to ask the following question, for a given part $p \in P$, a given relation $E$ and a given collection of good sets $B_1, \dots, B_{\arity(E)-1}$: What is the probability that the statement ``the fraction of elements of $p$ which disagree with $A$ on the value of $E$ with respect to $B_1, \dots, B_{\arity(E)-1}$ is greater than $\varepsilon + \zeta$'' is true?

Now, \cref{Proposition: Bound on the cap of the hypergeometric} tells us that not only is this probability small, but even if we were to ask polynomially many such questions, the probability that any of them would hold is (asymptotically) small. But we also know by \cref{Lemma: Polynomial bound on number of bad sets} that there exist only polynomially many such questions, hence the probability that any of them hold is (asymptotically) small. But if none of the questions holds of $p$ then we know $p$ is $(\varepsilon+ \zeta)$-excellent, which was our goal. We will now make this precise.

\begin{proposition}
\label{Proposition: Random partition of excellent set}
Consider a population with $N$ elements. Let $M_0, \dots, M_{k}$ be subsets of the population where $k = C N^{\ell}$ for constants $C$ and $\ell$, and suppose that $r$ divides $N$. Then for any $t > 0$, so long as $r\log r  + \log C  < 2t^2 N - r\ell \log N$,
	there is an equitable partition of $N$ into $r$ parts such that for each part $X$ of the partition, we have
\[
\frac{|M_i\cap X|}{|X|} \leq \frac{|M_i|}{N} + t
\]
	whenever $0 \le i \le k$.
\end{proposition}
\begin{proof}
By \cref{Proposition: Bound on the cap of the hypergeometric},
\[
\Pr\left[\bigvee_{ i \leq k}\left(\frac{H(N/r, N, M_i)}{N/r} \geq \frac{|M_i|}{N} + t\right)\right] \leq C \cdot N^\ell \cdot e^{-2t^2 N/r}.
\]
If $P$ is a uniformly random partition then for any $p \in P$ and $i \leq k$, the probability that $p$ contains at least $h$ many elements in $M_i$ is $\Pr[H(N/r, N, M_i) \geq h]$. 
Hence we have
\[
\Pr\left[\bigvee_{p \in P}\bigvee_{ i \leq k}\left(\frac{|p \cap M_i|}{|p|} \geq \frac{|M_i|}{N} + t\right)\right] \leq r \cdot C \cdot N^\ell \cdot e^{-2t^2 N/r}.
\]
But if $r\log r  + \log C   < 2t^2 N - r\ell \log N$, we then have 
	\[
		\Pr\left[\bigvee_{p \in P}\bigvee_{ i \leq k}\left(\frac{|p \cap M_i|}{|p|} \geq \frac{|M_i|}{N} + t\right)\right] < 1,
		\]
	and so there must be some such partition $P$ of $N$. 
\end{proof}

Putting these all together we have the following. 

\begin{proposition}
\label{Dividing epsilon-excellent classes into uniform excellent classes}
Let $\varepsilon$, $\zeta > 0$.  Suppose $A$ is an $\varepsilon$-excellent class, and $r\in \Nats$ is such that $r$ divides $|A|$.  Further, suppose
\[
	r\log r  + \log(2|\Lang|(q_\Lang!)) < 2\zeta^2 |A| - r 2^{\brap+1}\log|A|.
\]
Then there is an equitable partition of $A$ into $r$ parts, each of which is $(\varepsilon+\zeta)$-excellent. 
\end{proposition}
\begin{proof}
Let $M_0, \dots, M_k$ be sets of the form 
\[
	\{a_0 \in A\st \M\models E(a_{\sigma(0)}, \dots, a_{\sigma(\ell-1)})\}
\]
or of the form 
\[
	\{a_0 \in A\st \M\models \neg E(a_{\sigma(0)}, \dots, a_{\sigma(\ell-1)})\}
\]
for some $E \in \Lang$, some $a_1, \dots, a_{\ell-1} \in M$, and some permutation $\sigma$ of $\{0, \ldots, \ell-1\}$, where $\ell \defas \arity(E)$. 
	Then by
\cref{Lemma: Branching property from order property},
	$\M$ has the non-$2^{\brap+1}$ order property.
	Hence by \cref{Lemma: Polynomial bound on number of bad sets}, we have $k \leq 2 |\Lang|\cdot q_\Lang! \cdot |A|^{2^{\brap+1}}$. The result then follows immediately from 
	\cref{(epsilon k)-good by proxy}
	and
	\cref{Proposition: Random partition of excellent set}.
\end{proof}

\begin{proposition}
\label{Proposition: Equitable excellent partition}

Let $\zeta > 0$. 
	Suppose $0 < \varepsilon < 2^{-\brap} \cdot n_\Lang^{-1}$, and that
\begin{itemize}
\item[(a)] $\M$ does not have the $\brap$-branching property,

\item[(b)] $g \defas \lceil 5 \cdot n_\Lang \cdot \brap \cdot \log \brap \rceil$,

\item[(c)] $m$ is a positive natural number such that $m \cdot \lfloor\frac{1}{\varepsilon}\rfloor^g \leq |M|$, and

\item[(d)] $2\zeta^2 m - \frac{|M|}{m}\, {2^{\brap+1}}\log m   > 
	\frac{|M|}{m}\,\log \frac{|M|}{m} + \log (2|\Lang|(q_\Lang!))$.
\end{itemize}

Then there is a subset $M^+ \subseteq M$ and a partition $P$ of $M^+$ such that 

\begin{itemize}
\item[(i)] $|M\setminus M^+| <  m \cdot \lfloor\frac{1}{\varepsilon}\rfloor^g$,

\item[(ii)] each part of $P$ is $(\varepsilon+\zeta)$-excellent,

\item[(iii)] $P$ is equitable, and

\item[(iv)] each part of $P$ has size $m$.
\end{itemize}

\end{proposition}
\begin{proof}
Let $m_g = m$ and let $m_{i - 1} = m_i \cdot \lfloor\frac{1}{\varepsilon}\rfloor$ for $1\le i \le g$. By assumption (c) we have that $|M| \geq m_0$. Using assumptions (a) and (b) we can apply \cref{Proposition: Partition of excellent sets} to get a $M^+ \subseteq M$ and $P^+$ which satisfies (i), where each part of $P^+$ is $\varepsilon$-excellent, and where $m$ divides the size of each part of $P^+$. Note that the size $r$ of the partition $P^+$ is bounded above by $\frac{M}{m}$ and the size of any such partition is bounded below by $m$. Hence by applying (d), 
	we obtain
\[
	2\zeta^2 |p| - r{2^{\brap+1}}\log |p|   > r\log r  + \log(2|\Lang|(q_\Lang!))
\]
	for any part $p \in P^+$,
and so we can apply \cref{Dividing epsilon-excellent classes into uniform excellent classes} to find a refinement $P$ of $P^+$ which is equitable and where every part is $(\varepsilon+\zeta)$-excellent. 
\end{proof}

Finally, now that we have an equitable partition of a large subset of our graph, each of whose parts is appropriately excellent, we are able to prove one of our main results. 

\begin{theorem}
\label{Theorem: Stable Regularity for Relational Structures with choice of partition size}
Let $\zeta, \eta > 0$ and let $m \defas \lceil|M|\cdot \eta\rceil > 2$. 
	Suppose $0 < \varepsilon < 2^{-\brap} \cdot n_\Lang^{-1}$, and that
\begin{itemize}
\item[(a)] $\M$ does not have the $\brap$-branching property, 

\item[(b)] $g \defas \lceil 5 \cdot n_\Lang \cdot \brap \cdot \log \brap \rceil$,

\item[(c)] $\beta \defas \varepsilon^g - (\eta  + \frac{1}{|M|}) > 0$, and

\item[(d)] $2\zeta^2\eta m - {2^{\brap+1}}\log m   > \eta\log(2|\Lang|(q_\Lang!)) - \log \eta$.
\end{itemize}

Then there is an $\Lang$-structure $\N$ with the same underlying set $M$ as $\M$ and an equitable partition $P^*$ of $\N$ such that for all $E \in \Lang$,

\begin{itemize}
\item for all $\<p^*_i\>_{i < \ell} \subseteq P^*$, 
\[
\left|(E^\M \SymDiff E^\N) \cap \prod_{i < \ell} p^*_i\right| \leq \ell \cdot   \left(\frac{(\varepsilon + \zeta) \cdot \beta + \eta}{\beta}\right) \cdot \prod_{i < \ell}|p_i|,
\]
    
\item $P^*$ is indivisible,  and

\item $\frac{\beta}{\varepsilon^g\cdot \eta} \leq |P^*| \leq \frac{1}{\eta} +1$,
\end{itemize}
	where $\ell \defas \arity(E)$.
\end{theorem}
\begin{proof}
	First note that by (d) and the fact that $\frac{|M|}{m} \leq \eta^{-1}$, condition (d) of \cref{Proposition: Equitable excellent partition} holds. Next, $m \cdot \lfloor\frac{1}{\varepsilon}\rfloor^g \leq \lceil|M| \cdot \eta \rceil \frac{1}{\varepsilon^g} \leq (|M| \cdot \eta + 1) \frac{1}{\varepsilon^g} = |M| \cdot \frac{\eta + \frac{1}{|M|}}{\varepsilon^g} \leq  |M|$ and we so we can find a subset $M^+$ and an equitable partition $P^+$ of $M^+$ as in \cref{Proposition: Equitable excellent partition} where $|M \setminus M^+| < m \cdot \floor{\frac{1}{\varepsilon}}^g$ and each part of $P^+$ has size $m$. 

As each part of $P^+$ is $(\varepsilon+\zeta)$-excellent, by \cref{Changes in edges due to excellence needed to get indivisible} there is a structure $(M^+, E^{**})$ on the same underlying set as $M^+$ such that $P^+$ is indivisible and 
	$|(E^\M|_{M^+} \SymDiff E^{**}) \cap \prod_{i < \ell} p_i| \leq \ell \cdot (\varepsilon + \zeta) \cdot  \prod_{i < \ell} |p_i|$
	for all $p_0, \dots, p_{\ell-1} \in P$.

Finally, we can extend $P^+$ to an equitable partition $P^*$ of $M$ by adding elements of $M \setminus M^+$ arbitrarily while preserving the appropriate sizes of the parts of $P$. As $|M \setminus M^+| < m \cdot \floor{\frac{1}{\varepsilon}}^g$, we have 
\begin{align*}
|P^*| &\geq \frac{|M| - \lceil|M|\cdot \eta\rceil \cdot \floor{\frac{1}{\varepsilon}}^g}{\lceil|M|\cdot \eta\rceil} 
\geq \frac{|M| - (|M|\cdot \eta +1)\cdot \floor{\frac{1}{\varepsilon}}^g}{|M|\cdot \eta} \\
&= \frac{1 - (\eta + \frac{1}{|M|}) \cdot \floor{\frac{1}{\varepsilon}}^g}{\eta} 
\geq
\frac{1 - (\eta + \frac{1}{|M|}) \cdot (\frac{1}{\varepsilon})^g}{\eta} \\
&=
\frac{\varepsilon^g - (\eta + \frac{1}{|M|})}{\varepsilon^g \cdot \eta} 
=
\frac{\beta}{\varepsilon^g \cdot \eta}.
\end{align*}
Also note that each part of $P^*$ has size at least $m$, and so $|P^*| \leq \frac{|M|}{m} \leq \frac{1}{\eta}+1$. 

Further note that by an appropriate assignment of edge relations on $M\setminus M^+$, we can extend $E^{**}$ to an edge relation $E^\N$ such that $P^*$ is also an indivisible partition of $\N$. Let 
\[
	k^* \defas \sup\left\{\frac{|p^* \setminus p|}{|p|} \st p \in P^+, \ p^* \in P^*, \textrm{~and~} p \subseteq p^*\right\}.
\]
Then we have 
\begin{align*}
k^* &\leq \frac{\frac{m \cdot \floor{\frac{1}{\varepsilon}}^g}{|P^*|}}{ m}
=\frac{\floor{\frac{1}{\varepsilon}}^g}{|P^*|}
\leq \frac{(\frac{1}{\varepsilon})^g}{\frac{\beta}{\varepsilon^g \cdot \eta}}
= 
\frac{\eta }{\beta}.
\end{align*}

Let $X_0$ be the collection of $\ell$-tuples at least one element of which is contained in $M\setminus M^+$. Suppose $p_0^*, \dots, p_{\ell-1}^* \in P^*$. We then have 
\begin{align*}
	\Bigl|\bigl((E^\M \cap X_0) \,\SymDiff\, (E^\N \cap X_0)\bigr) \,\cap \,\prod_{i < \ell} p_i^*\Bigr| & \leq\bigl|X_0 \cap \prod_{i < \ell} p_i^*\bigr| \\
&\leq \ell \cdot k^* \cdot \prod_{i < \ell}|p_i^*|\\
&\leq \ell \cdot \frac{\eta}{\beta}
 \cdot \prod_{i < \ell}|p_i^*|.
\end{align*}

Putting this together we get 
\begin{align*}
\Bigl|(E^\M \SymDiff E^\N ) \cap \prod_{i < \ell} p_i^*\Bigr| &\leq \ell \cdot (\varepsilon + \zeta) \cdot  \prod_{i < \ell}|p_i^*| + \ell \cdot \frac{\eta}{\beta}
\cdot \prod_{i < \ell}|p_i^*| \\
&\leq \ell \cdot  \left(\varepsilon + \zeta +\frac{\eta}{\beta}\right) \cdot \prod_{i < \ell}|p_i^*| \\
&\leq \ell \cdot  \left(\frac{(\varepsilon + \zeta) \cdot \beta + \eta}{\beta}\right) \cdot \prod_{i < \ell}|p_i^*|.
\end{align*}
\end{proof}

There is a tension among the three parameters $\varepsilon$, $\eta$, and $\zeta$. Namely, as $\eta$ becomes smaller, the potential size of the partition becomes larger, but at the same time, the fraction of elements that we need to change becomes smaller. On the other hand, as $\varepsilon$ becomes smaller, both the potential partition size and the number of elements we need to change become larger. Finally, $\zeta$ must be chosen to as to be consistent with the other two parameters in (d); in particular, as $\eta$ becomes smaller, $\zeta$ must get larger.

While \cref{Theorem: Stable Regularity for Relational Structures with choice of partition size} provides precise lower bounds on how large a structure we need in order for stable regularity to come into play, these bounds can be unwieldy. If instead we are willing to simply consider ``sufficiently large'' structures then the result has a much cleaner form.

\begin{theorem}[Stable regularity for finite relational structures]
\label{Theorem: Stable Regularity for Relational Structures}
	Let $\M$ be a finite $\Lang$-structure with underlying set $M$, and 
define $g \defas \lceil 5 \cdot n_\Lang \cdot \brap \cdot \log \brap \rceil$. 
	Suppose $0 < \varepsilon < 2^{-(g+1)(g+2)}$.
	Then there is some $k_\varepsilon$ such that if
	$|M|\geq k_\varepsilon$ and
	$\M$ does not have the $\brap$-branching property, 
then there is an $\Lang$-structure $\N$ with the underlying set $M$, and an equitable partition $P$ of $\N$, such that 
	for all $E \in\Lang$,
\begin{itemize}
\item for all $\<p_i\>_{i < \ell} \subseteq P$,
\[
\Bigl|(E^\M \SymDiff E^\N) \cap \prod_{i < \ell} p_i\Bigr| \leq \ell \cdot \varepsilon \cdot \prod_{i < \ell}|p_i|,
\]

\item $P$ is indivisible, and

\item $|P| \leq \varepsilon^{-g-2}$,
\end{itemize}
	where $\ell \defas \arity(E)$.
\end{theorem}
\begin{proof}
	Suppose $0 < \varepsilon < 2^{-(g+1)(g+2)}$.
	We will choose $\varepsilon_1$, $\zeta_1$, $\eta_1>0$ and $k_\varepsilon\in \Nats$ in terms of $\varepsilon$ such that for all $\M$ with the 
	non-$\brap$-branching property and
	$|M|\geq k_\varepsilon$,
	we may apply 
	\cref{Theorem: Stable Regularity for Relational Structures with choice of partition size} to $\varepsilon_1$, $\zeta_1$, and $\eta_1$ to produce an $\Lang$-structure $\N$ and equitable partition $P$, which we will verify have the desired properties.

Choose $\gamma_1$ such that $1< \gamma_1 < 2$ and let $p>4$ be such that 
	$\gamma_1^g(1 + \varepsilon) < p < 2^{g+1}$ (which is possible as $g>1$, as $\gamma_1 < 2$, and as $\varepsilon <1$).
Therefore 
\[
\gamma_1^g < p-\varepsilon \cdot \gamma_1^g
\]
and so 
\[
\frac{\gamma_1^g}{p-\varepsilon \cdot \gamma_1^g}  < 1.
\]
But then we also have have 
	\[
\frac{\gamma_1^g}{1-\frac{\varepsilon}{p+1}  \cdot \gamma_1^g} 
< 
\frac{\gamma_1^g}{1-\frac{\varepsilon}{p} \cdot \gamma_1^g}
= 
p \frac{\gamma_1^g}{p-\varepsilon \cdot \gamma_1^g}
< 
p.
	\tag{A}
\label{eqn:A}
	\]

Choose
$\varepsilon_1 = \frac{\varepsilon}{(p+2)\gamma_1}$.
	In particular, we have
$\varepsilon_1 \cdot \gamma_1 < \frac{\varepsilon}{p+1}< 1$.
	Further, as $p>1$ and $\gamma_1 > 1$, we have
	$\varepsilon_1 < 	\frac{1}{1 + \gamma_1^{g+1}}$,
	and  so
$\varepsilon_1(1 + \gamma_1^{g+1}) < 1$.

Let $\zeta_1 \defas \varepsilon_1 \cdot (\gamma_1 - 1)$,
so that
$\gamma_1 = 1 + \frac{\zeta_1}{\varepsilon_1}$.

Let $\eta_1 \defas  \varepsilon_1^{g+1} \cdot \gamma_1^{g+1}
=
(\varepsilon_1 + \zeta_1)^{g+1}$.

Let $\beta \defas \varepsilon_1^g - (\eta_1 + \frac{1}{|M|})$.

	Let $k_\varepsilon$ be large enough that 
\begin{itemize}
\item[(1)] $k_\varepsilon\eta_1> 2$,

\item[(2)] $
2 k_\varepsilon
		\zeta_1^2\eta_1^2 - {2^{\brap+1}}\log \bigl(k_\varepsilon\eta_1\bigr)   > \eta_1\log(2|\Lang|(q_\Lang!)) - \log \eta_1$,

	\item[(3)] $k_\varepsilon > \frac{2^{\brap+1}}{2 \zeta_1^2 \eta_1^2}$,

\item[(4)] $\frac{\gamma_1^g}{1-\frac{\varepsilon}{p+1}  \gamma_1^g - \frac{1}{\varepsilon_1^g k_{\varepsilon}}} < p$, and

\item[(5)] $k_\varepsilon > \varepsilon_1^{-g-1}$.
\end{itemize}
(Any sufficiently large $k_\varepsilon$ satisfies (4) by \eqref{eqn:A}, and clearly (1), (2), (3), and (5) hold for all sufficiently large $k_\varepsilon$.)

	Let $m \defas  \lceil|M|\cdot \eta_1\rceil$ and let
	$\ell \defas \arity(E)$.
	We have assumed that $\M$ does not have the $\brap$-branching property. 
We now show that 
$m > 2$, that
$\beta >0$, and that
$2\zeta_1^2\eta_1 m - {2^{\brap+1}}\log m   > \eta_1\log(2|\Lang|(q_\Lang!)) - \log \eta_1$
	(so that we may apply \cref{Theorem: Stable Regularity for Relational Structures with choice of partition size}).

Note that (1) ensures that $m = \lceil|M|\cdot \eta_1\rceil > 2$. 
The function 
$2\zeta_1^2\eta_1 x - {2^{\brap+1}}\log x$ is increasing for $x > \frac{2^{\brap+1}}{2 \zeta_1^2 \eta_1}$, and so 
(2) and (3) imply that 
\[
	2\zeta_1^2\eta_1 m - {2^{\brap+1}}\log m   > \eta_1\log(2|\Lang|(q_\Lang!)) - \log \eta_1
\]
holds. 

Now
\[
\beta = \varepsilon_1^g - (\varepsilon_1\gamma_1)^{g+1} - \frac{1}{|M|} \geq \varepsilon_1^g - (\varepsilon_1\gamma_1)^{g+1} - \frac{1}{k_\varepsilon} > 
\varepsilon_1^g - (\varepsilon_1\gamma_1)^{g+1} - \varepsilon_1^{g+1},
\]
where the last inequality follows from (5). 
But 
\[
\varepsilon_1^g - (\varepsilon_1\gamma_1)^{g+1} - \varepsilon_1^{g+1} = (\varepsilon_1)^g (1-\varepsilon_1(\gamma_1^{g+1} +1)). 
\]
Recall that
$1 > \varepsilon_1(\gamma_1^{g+1} +1)$, 
and so $\beta > 0$.  
	Also note that (iv) implies $\varepsilon < 2^{-\brap} \cdot n_\Lang^{-1}$, and so $\varepsilon_1< 2^{-\brap} \cdot n_\Lang^{-1}$.

Hence we may apply
	\cref{Theorem: Stable Regularity for Relational Structures with choice of partition size} to $\varepsilon_1$, $\zeta_1$, and $\eta_1$ to obtain an $\Lang$-structure $\N$ 
	with the same underlying set $M$ as $\M$ and an equitable partition $P$ of $\N$ such that for all $E \in \Lang$,

\begin{itemize}
\item for all $\<p^*_i\>_{i < \ell} \subseteq P$, 
\[
\left|(E^\M \SymDiff E^\N) \cap \prod_{i < \ell} p^*_i\right| \leq \ell \cdot   \left(\frac{(\varepsilon_1 + \zeta_1) \cdot \beta + \eta_1}{\beta}\right) \cdot \prod_{i < \ell}|p_i|,
\]
    
\item $P$ is indivisible,  and

\item $|P| \leq \frac{1}{\eta_1} +1$.
\end{itemize}

We must show that 
$\frac{(\varepsilon_1 + \zeta_1) \cdot \beta + \eta_1}{\beta}
\le \varepsilon$
and that $\frac{1}{\eta_1} +1 \le \varepsilon^{-g-2}$.

Recall that
$\varepsilon_1 + \zeta_1  = 
\varepsilon_1\gamma_1 $.
Observe that
\begin{eqnarray*}
\frac{\eta_1}{\beta}
	&=&
	\frac{(\varepsilon_1 \gamma_1)^{g+1}}{\beta}\\
	&=&
	\frac{(\varepsilon_1 \gamma_1)^{g+1}}{\varepsilon^g - (\varepsilon_1 \gamma_1)^{g+1} - \frac{1}{|M|}} \\
	&\leq &
	\frac{(\varepsilon_1 \gamma_1)^{g+1}}{\varepsilon^g - (\varepsilon_1 \gamma_1)^{g+1} - \frac{1}{k_\varepsilon}} \\
	&=& 
	\varepsilon_1 \cdot \gamma_1\frac{\gamma_1^{g}}{1 - (\varepsilon_1 \gamma_1) \gamma_1^{g} - \frac{1}{\varepsilon_1^g k_\varepsilon}} \\
	&<& 
\varepsilon_1\gamma_1 p, 
\end{eqnarray*}
where the last inequality follows from (4).
Hence 
$\frac{(\varepsilon_1 + \zeta_1) \cdot \beta + \eta_1}{\beta}
= \varepsilon_1 \gamma_1 +  \frac{\eta_1}{\beta}
< \varepsilon_1 \gamma_1 + \varepsilon_1\gamma_1 p
< \varepsilon$.

Now, we have
\[
\frac{1}{\eta_1} 
 = 
\frac{1}{(\varepsilon_1 \gamma_1)^{g+1}} 
=
\bigl(\frac{p+2}{\varepsilon}\bigr)^{g+1} 
<
(2p)^{g+1}\varepsilon^{-g-1} - 1
\]
as $p>4$.
Finally, we have
\[
\frac{1}{\eta_1} +1
<
(2p)^{g+1}\varepsilon^{-g-1} 
<
(2 \cdot 2^{g+1})^{g+1}\varepsilon^{-g-1} 
\leq  2^{(g+1)(g+2)}\varepsilon^{-g-1}
< \varepsilon^{-g-2},
 \]
where the last inequality follows because
	$\varepsilon < 2^{-(g+1)(g+2)}$.
\end{proof}

Note that the corresponding counting and removal lemmas follow immediately from \cref{Theorem: Stable Regularity for Relational Structures}.

\section{Almost stable regularity for relational structures}

We now consider structures that are not stable, but which have very few witnesses to their non-stability. In this ``almost stable'' situation we will show that there is also a highly structured regularity lemma, in which a modification of the original structure arises as a finite blow-up.
However, in this almost stable case, we merely get a \emph{global} regularity lemma, rather than a \emph{local} one. 

More precisely,
instead of obtaining a blow-up by changing a small fraction of the relations across each tuple of parts of the partition (of appropriate length), we can instead obtain a blow-up only by changing a small fraction of the relations across the entire structure. The key difference is that the vertices corresponding to these modified relations might be concentrated in certain regions of the structure, in which they make up a large fraction of the vertices.

This distinction between local and global regularity is often referred to as the distinction between \emph{regularity} and \emph{weak regularity}.

\begin{definition}
	\label{homdensity-finite}
Let $\M$ and $\N$ be finite $\Lang$-structures with underlying sets $M$ and $N$ respectively, and set $n = |N|$ and $k=|M|$. Define the \defn{induced homomorphism density} of $\M$ in $\N$ to be 
\[
\tind(\M, \N) \defas \frac{\bigl| \ind(\M, \N) \bigr|}
	{n(n-1)\cdots(n-k+1)},
	\]
	where $\ind(\M, \N)$ is the number of embeddings from $\M$ to $\N$, in other words, injective homomorphisms that yield an induced substructure (i.e., which preserve all relations and all negations of relations).
\end{definition}

For more details on induced homomorphism densities in the case of graphs, see \cite[\S5.2]{MR3012035}; for a more general setting, see 
\cite[\S2]{2014arXiv1412.8084A} and
\cite[Chapter~1]{kruckman-thesis}. 

\begin{definition}
	Let $\brap \in \Nats$. An $\Lang$-structure $\M$ \defn{minimally} has the $\brap$-branching property for a quantifier-free formula $\varphi(\x; \y)$ if $\M$ has the $\brap$-branching property for $\varphi(\x;\y)$ and no induced substructure of $\M$ has the $\brap$-branching property for $\varphi(\x;\y)$. 
\end{definition}

\begin{lemma}
\label{Upper bound on size of structures which are minimal for branching} 
If $\M$ minimally has the $\brap$-branching property for $\varphi(\x;\y)$ then $|\M| \leq 2^{\brap} \cdot (|\x|+|\y|)$. 
\end{lemma}

\begin{proof}
	Suppose $\M$ has the $\brap$-branching property for $\varphi(\x;\y)$ but $|\M| > 2^{\brap} \cdot (|\x|+|\y|)$.
	Let $M_0 \subseteq M$ consist of all tuples in a witness to the
	$\brap$-branching property for $\varphi(\x;\y)$. Then
	$|M_0| \leq 2^{\brap} \cdot (|\x|+|\y|)$, and so $\M_0$, the induced substructure of $\M$ with underlying set $M_0$, is a proper substructure of $\M$.
	Hence $\M_0$ also has the $\brap$-branching property for $\varphi(\x;\y)$, and so $\M$ was not minimal.
\end{proof}

\begin{definition}
Let $\brap \in \Nats$ and $\delta > 0$. An $\Lang$-structure $\M$ has the $(\delta, \brap)$-branching property for a quantifier-free formula $\varphi(\x; \y)$ if there is a structure $\N$ which minimally has the $\brap$-branching property and for which $\tind(\N, \M) \geq \delta$.  

We say an $\Lang$-structure $\M$ has the $(\delta, \brap)$-branching property if it has the $(\delta, \brap)$-branching property for some relation $E \in \Lang$ with some partition of the variables where one part is a singleton.
\end{definition}

Note that a structure $\M$ has the $\brap$-branching property for a quantifier-free formula $\varphi(\x;\y)$ exactly when there is a structure $\N$ which minimally has the $\brap$-branching property for $\varphi(\x;\y)$ and for which there exists at least one embedding from $\N$ into $\M$. This motivates the idea that a structure not having the $(\delta,\brap)$-branching property is a sign that it has very few witnesses to non-stability.

The next result follows from \cite[Theorem 2]{2014arXiv1412.8084A}.

\begin{proposition}[{\cite[Theorem 2]{2014arXiv1412.8084A}}]
\label{Removal Lemma}
Suppose $\<\F_i\>_{i \leq \ell}$ is a finite collection of finite $\Lang$-structures. Then for every $\varepsilon > 0$ there is an $n_\varepsilon \in \Nats$ and a $\delta >0$ such that whenever
\begin{itemize}
    \item $\M$ is a finite $\Lang$-structure with $|M| > n_\varepsilon$ and

    \item $\tind(\F_i, \M) < \delta$ for all $i \leq \ell$,
\end{itemize}
then there is an $\Lang$-structure $\M^*$ with the same underlying set as $\M$ such that 
\begin{itemize}
\item $\tind(\F_i, \M^*) = 0$ for all $i \leq \ell$ and

\item 
		$|E^{\M} \SymDiff E^{\M^*}| \leq \varepsilon \cdot |M|^{\arity(E)}$
	for all $E \in \Lang$.
\end{itemize}
\end{proposition}

Note that \cite[Theorem 2]{2014arXiv1412.8084A} was originally stated in terms of quantities of the form $p(\F_i,\M)$ (and analogously for $\M^*$), which equals $\tind(\F_i,\M)/\tind(\F_i, \F_i)$ (by their Fact 1). Note that when $\F_i$ minimally has the 
	$\brap$-branching property for all $E\in\Lang$ with partitions of the variables where one part is a singleton, then the
denominator $\tind(\F_i, \F_i)$ is bounded by
$ (2^{\brap} \cdot q_\Lang)^ {2^{\brap} \cdot q_\Lang} $  
by
\cref{Upper bound on size of structures which are minimal for branching}.
Hence one can check that 
the removal lemma \cref{Removal Lemma} is essentially equivalent to theirs.

\begin{theorem}[Almost stable regularity for finite relational structures]
\label{Almost Stable Regularity for Relational Structures}
Let $g \defas \lceil 5 \cdot n_\Lang \cdot \brap \cdot \log \brap \rceil$. For all $\varepsilon > 0$ there is a $k_\varepsilon$  and $\delta > 0$ such that if
\begin{itemize}
\item $\M$ is a $\Lang$-structure with $|\M| \geq k_\varepsilon$, and
\item $\M$ does not have the $(\delta, \brap)$-branching property,
\end{itemize}
then there is a structure $\N$ with the same underlying set as $\M$ and an equitable partition $P$ of $\N$ such that 
\begin{itemize}
\item[(i)] $|E^\M \SymDiff E^\N| \leq \arity(E)  \cdot\varepsilon\cdot |M|^{\arity(E)}$
	for all $E \in \Lang$,
    
\item[(ii)] $P$ is indivisible, and
    
\item[(iii)] $|P| \leq \bigl({\frac{\varepsilon}{2}}\bigr)^{-g-2}$.
\end{itemize}
\end{theorem}
\begin{proof}
	First apply \cref{Removal Lemma} with $\frac{\varepsilon}{2}$ to get a structure $\M^*$ without the $\brap$-branching property such that
	$|E^\M \SymDiff E^{\M^*}| \leq \arity(E)  \cdot\frac{\varepsilon}{2}\cdot |M|^{\arity(E)}$ 
	for all $E \in \Lang$.
	Then apply \cref{Theorem: Stable Regularity for Relational Structures} with $\M^*$ and $\frac{\varepsilon}{2}$ to get a structure $\N$ and partition $P$ such that  (ii) and (iii) hold and
	$|E^{\M^*} \SymDiff E^{\N}| \leq \arity(E)  \cdot\frac{\varepsilon}{2}\cdot |M|^{\arity(E)}$ 
	for all $E \in \Lang$. Then condition (i) follows by considering the symmetric difference of $E^\M$ and $E^\N$.
\end{proof}

\section{Borel stable regularity for relational structures}

We now consider ways of extending the almost stable regularity lemma from finite relational structures to Borel relational structures.
Somewhat analogously for the case of graphs, 
Lov\'asz and Szegedy \cite{MR2306658}
have developed analytic versions of the graph regularity lemma, expressed in terms of \emph{graphons} and measurable partitions of their domains.

In this section we provide an almost stable regularity lemma for Borel structures, which shows that every Borel structure that is almost stable (in a sense we make precise) is close in $L^1$ to a Borel blow-up of a finite structure.

We will define Borel structures to have underlying set $[0,1]$, and we will mostly deal with Lebesgue measure $\lambda$ on $[0,1]$.
Note that whenever $(P, \mu)$ is a standard probability space, there is a measure preserving map from $([0,1], \lambda)$ onto $(P, \mu)$. 
Hence the main arguments of this section go through with $([0,1], \lambda)$ replaced by an arbitrary standard probability space.

We begin with definitions of
Borel structures and the notions of $L^1$-distance, blow-ups, and induced homomorphism densities for them. These can be seen as analogous to the corresponding notions for the theory of graphons \cite[Chapter~7]{MR3012035}. 

\begin{definition}
\label{Borel L-structure}
	A \defn{Borel} $\Lang$-structure $\M$ is an $\Lang$-structure with underlying set $[0,1]$ such that for all $E\in \Lang$,
	the relation $E^\M$ interpreting the relation symbol $E$ is Borel.
\end{definition}

It will often be convenient to work with characteristic functions instead of relations.

\begin{definition}
\label{functional-structure}
Let $\M$ be an $\Lang$-structure (with arbitrary underlying set). 
For each $E \in \Lang$, define
	$\widetilde{E}^\M \colon [0,1]^{\arity(E)} \to \{0,1\}$ to be the characteristic function of the relation $E^\M$. Note that these functions are Borel when $\M$ is a Borel $L$-structure.
\end{definition}

The $L^1$-distance plays a key role in our arguments in this section.

\begin{definition}
\label{L-1 distance}
Suppose $\M$ and $\N$ are Borel $\Lang$-structures. We define the \defn{$L^1$-distance} between $\M$ and $\N$, written $\dist(\M, \N)$, to be 
\[
	\sum_{E \in \Lang}\int_{[0,1]^{\arity(E)}}\Bigl|\widetilde{E}^{\M}(\boldx) - \widetilde{E}^{\N}(\boldx)\Bigr| \dee \boldx,
\]
where $\boldx$ is a tuple of variables of length $\arity(E)$.
\end{definition}

We now consider finite structures, and their relationship to Borel structures via Borel blow-ups.
All finite structures in this section will have underlying set an initial segment of $\Nats$.

Every finite $\Lang$-structure with counting measure induces a Borel $\Lang$-structure, by taking its \emph{Borel blow-up}.
For each $k$ such that $0\leq k < r-1$, define $\iota_r(k) \defas [\frac{k}{r}, \frac{k+1}{r})$ and $\iota_r(r-1) \defas [\frac{r-1}{r}, 1]$.

\begin{definition}
	\label{def:blowup}
	Suppose $\M$ is a finite $\Lang$-structure with underlying set $M$. Define its \defn{Borel blow-up}, $\Borel{\M}$, to be the Borel $\Lang$-structure such that for all $E \in \Lang$ and $i < \arity(E)$, whenever $x_i \in \iota_r(k_i)$ for all $k_i < |M|$  we have 
\[
	\Borel{\M}\models E(x_0, \dots, x_{\arity(E)-1}) \quad \text{ if and only if } \quad \M \models E(k_0, \dots, k_{\arity(E)-1}).
\]
\end{definition}

Observe that the Borel blow-up of a finite structure is a particular kind of blow-up, in the sense of \cref{gen-blowup}.

By a standard argument, every Borel $\Lang$-structure is close in $L^1$ to the Borel blow-up of some finite $\Lang$-structure.

\begin{lemma}
\label{Approximate Borel structures by finite structures structures in L1}
	Let $\M$ be a Borel $\Lang$-structure. For all $\varepsilon > 0$ and all $n_0 \in \Nats$, there is an $n > n_0$ and an $\Lang$-structure $\N$ with underlying set $\{0, \ldots, n-1\}$ such that 
	$\dist(\M, \Borel{\N}) < \varepsilon$. 
\end{lemma}
\begin{proof}
	There is some $n\in\Nats$ such that
	for every $E\in \Lang$, 
	some set $S_{E, \varepsilon}\subseteq [0,1]^{\arity(E)}$  
	that is a finite union of sets of the form $\prod_{s <\arity(E)} \iota_{n}(k_s)$ 
	satisfies
	$\lambda(E \SymDiff S_{E, \varepsilon}) < \varepsilon / |\Lang|$.

	Let $\N$ be the $\Lang$-structure with underlying set $\{0, \ldots, n-1\}$ satisfying
	\[ \N \models E(k_0, \ldots, k_{\arity(E)-1})
	\quad \text{ if and only if } \quad 
	\textstyle
	\prod_{s <\arity(E)} \iota_{n}(k_s) \subseteq S_{E, \varepsilon}.
		\]
	for $E\in \Lang$ and $k_0, \ldots, k_{\arity(E)-1}< n$.
	By construction of $\N$, by summing over all relation symbols $E \in \Lang$, we have 
	$\dist(\M, \Borel{\N}) < \varepsilon$. 
\end{proof}

For finite structures of the same size (hence on the same underlying set, an initial segment of $\Nats$) with a single relation, their normalized edit distance is the same as their $L^1$-distance. This fact follows immediately from Definitions~\ref{L-1 distance} and 
	\ref{def:blowup} of $L^1$-distance and Borel blow-up.

\begin{lemma}
\label{Difference in edges related to L1}
Suppose $\M$ and $\M^*$ are finite $\Lang$-structures on the same underlying set $M$. Then 
\[
\dist(\Borel{\M}, \Borel{\M^*}) = \sum_{E \in \Lang}\frac{|E^{\M} \SymDiff E^{\M^*}|}{|M^{\arity(E)}|}.
\]
\end{lemma}

We will later need finite blow-ups to make a structure large enough so as to apply the results of earlier sections.
A finite blow-up can also be seen as an instance of \cref{gen-blowup}.

\begin{definition}
	\label{def:finite-blowup}
Let $\M$ be a finite $\Lang$-structure and let $p\in \Nats$ be positive.
	The \defn{$p$-fold blow-up} of $\M$ is defined to be the structure $\M_p$ of size $p \cdot |M|$
	such that 
	for each relation $E \in \Lang$ and $x_0, \dots, x_{\arity(E)-1} \in M_p$, the underlying set of $\M_p$, we have
		\[
			\textstyle
			\M_p \models E(x_0, \dots, x_{\arity(E)-1})
			\quad
			\text{if and only if}
			\quad
			\M \models E( \floor{\frac{x_0}{p}}, \dots, \floor{\frac{x_{\arity(E)-1}}{p}}).\]
		We call $\M_p$ a \defn{finite blow-up} of $\M$.
\end{definition}

It is immediate that replacing a finite structure by a finite blow-up does not change its Borel blow-up.

\begin{lemma}
\label{Blow-ups give the same Borel structures}
Suppose $\M_p$ is the $p$-fold blow-up of a finite $\Lang$-structure $\M$. Then $\Borel{\M_p} = \Borel{\M}$. 
\end{lemma}

We may define induced homomorphism densities for Borel $\Lang$-structures, similarly to \cref{homdensity-finite}.
For more details on an analogous notion for graphons, see \cite[\S7.2]{MR3012035}.

\begin{definition}
Suppose $\M$ is a finite $\Lang$-structure with underlying set $\{0, \dots, |\M|-1\}$ and $\N$ is a Borel $\Lang$-structure. We define the \defn{induced homomorphism density} of $\M$ in $\N$ to be 
\[
\tind(\M, \N) \defas \int_{I(\M, \N)} \dee\boldx 
\]
	where $I(\M, \N)$ is the set of embeddings from $\M$ to $\N$, considered as a Borel subset of $[0,1]^{|\M|}$.
\end{definition}

The following lemma is immediate. 

\begin{lemma}
\label{Subgraph densities as finite structures are bounded by the same as Borel Structures}
Let $\M$ and $\N$ be finite $\Lang$-structures.
Then
\[
\tind(\M, \N) \leq \tind(\M, \Borel{\N}) .
\]
\end{lemma}

In the case of Borel structures, we only ever care about a structure up to measure-zero sets. However, any stable Borel structure can be modified on a set of measure $0$ to make it unstable, and so we need to consider a weaker notion of stability for Borel structures.  
We use \cref{Subgraph densities as finite structures are bounded by the same as Borel Structures} to extend the definition of the $(\delta, \brap)$-branching property to Borel $\Lang$-structures.

\begin{definition}
Let $\brap \in \Nats$ and $\delta > 0$. A Borel $\Lang$-structure $\N$ has the $(\delta, \brap)$-branching property for a quantifier-free formula $\varphi(\x; \y)$ if there is a structure $\M$ which minimally has the $\brap$-branching property for $\varphi(\x;\y)$ and for which $\tind(\M, \N) \geq \delta$.  
\end{definition}

We can obtain a bound on the differences of induced homomorphism densities obtained from a bound on the $L^1$-distances of two structures.

\begin{lemma}
\label{Change in L1 doesn't change subgraph density}
Let $\F$ be a finite $\Lang$-structure with underlying set $F$, and let $\M$ and $\N$ be Borel $\Lang$-structures. If $\dist(\M, \N) \leq \varepsilon$ then 
\[
\bigl|\tind(\F, \M) - \tind(\F, \N)\bigr| \leq |\Lang| |F|^{q_\Lang} \cdot \varepsilon.
\]
\end{lemma}
\begin{proof}
	Let $\Ind_{X}$ denote the indicator function of a set $X$.
Observe that
	\[
\int_{[0,1]^{|F|}}
	\Bigl| \Ind_{I(\F, \M)}(\boldx)- \Ind_{I(\F, \N)}(\boldx)\Bigr| \dee \boldx 
	\, \leq\, \sum_{E\in\Lang} |F|^{\arity(E)} \cdot \varepsilon,
\]
where 
	$|\boldx| = |F|$.
	But $q_\Lang$ is the maximum arity of a relation symbol in $\Lang$, and so 
	$\sum_{E\in\Lang} |F|^{\arity(E)} \le |\Lang| |F|^{q_\Lang}$, as desired.
\end{proof}

\begin{definition}
	A partition of $[0,1]$ is \defn{Borel} if it is a countable partition each part of which is Borel.
	A Borel partition is \defn{equitable} if every part has the same Lebesgue measure.
\end{definition}

A Borel $\Lang$-structure with an equitable finite partition can be thought of as a Borel blow-up of a finite structure (up to measure-preserving isomorphism).

\begin{definition}
	Suppose $\M$ is a Borel $\Lang$-structure. A Borel partition $P$ of $[0,1]$ is \defn{indivisible} with respect to $\M$ if for all relations $E \in \Lang$, for all $p_0, \dots, p_{\arity(E)-1} \in P$, and for any pair of tuples $\<a_i^0\>_{i < \arity(E)}, \<a_i^1\>_{i < \arity(E)}$ such that $a_i^0, a_i^1 \in p_i$ for $i < \arity(E)$, we have
\[
	\widetilde{E}^\M(a_0^0, \dots, a_{\arity(E)-1}^0) = \widetilde{E}^\M(a_0^1, \dots, a_{\arity(E)-1}^1).
\]
\end{definition}

Whereas in equitable partitions of finite structures, the size of the parts can differ by up to $1$ (when the partition size does not divide the structure size), in the Borel case the Lebesgue measure of any two parts must be be equal.  The following lemma relates these two notions.

\begin{lemma}
\label{Find equitable indivisible partitions on [0 1] from finite ones}
	Let $\M$ be a finite $\Lang$-structure with underlying set $M$.
	Suppose $P$ is an indivisible partition of $\M$. Then there is a Borel $\Lang$-structure $\M^+$ and an equitable partition $P^+$ of $[0,1]$ such that
\begin{itemize}
\item $P^+$ is indivisible with respect to $\M^+$ and

\item $\dist(\Borel{\M}, \M^+) \leq \sum_{E \in \Lang}\frac{|P|-1}{|M|}$.
\end{itemize}
\end{lemma}
\begin{proof}
Let $r \defas \min\{|p|\st p \in P\}$. Let $A$ contain exactly $r$ elements from each $p \in P$. Note that $|M \setminus A| \leq |P| - 1$ as $P$ is equitable. For each $p \in P$ let $p^* \defas \bigcup_{a \in p \cap A} \iota_{|M|}(a)$. 

	Let $S$ be a partition of $[0,1]-\bigcup_{p\in P} p^*$ into $|P|$-many parts $\<s_p\>_{p \in P}$ of equal Lebesgue measure. For each $p\in P$, let $p^+ \defas p^* \cup s_p$. Define $P^+ \defas \{p^+ \st p \in P\}$. It is then immediate that $P^+$ is an equitable partition. 

	For the remainder of this proof, consider $E \in \Lang$, and 
	let $\ell \defas \arity(E)$;
	the result will follow by summing over all relation symbols in  $\Lang$. 
For every $p \in P$ 
	choose $x_p \in p$.
	For every $p^+_0, \ldots, p^+_{\ell-1} \in P^+$, and for every
	$y_0, \ldots, y_{\ell- 1} \in [0,1]$ such that 
	$y_i \in p^+_i$
	for all $i < \ell$, 
	let 
	\[
		\M^+ \models E(y_0, \ldots, y_{\ell-1} )
		\quad \text{if and only if}\quad
		\M \models E(x_{p_0}, \dots, x_{p_{\ell-1}}).
		\]
	Note that $P^+$ is indivisible with respect to $\M^+$. 

	Because $P$ was indivisible with respect to $\M$, the definition of $\M^+$ does not depend on the choice of the elements $x_p$. In particular this means $\widetilde{E}^{\M^+}|_{\iota_{|M|}(A)^{\ell}} = E^{\Borel{\M}}|_{\iota_{|M|}(A)^{\ell}}$. Finally, we have
\[
\lambda^{\ell}([0,1]^{\ell}\setminus \iota_{|M|}(A)^{\ell}) 
	\, \leq \,
	\lambda([0,1] \setminus \iota_{|M|}(A)) 
	\, \leq \,
	\frac{|P|-1}{|M|},
\]
which completes the argument for this particular $E\in \Lang$.
\end{proof}

\begin{theorem}[Almost stable regularity for Borel structures]
Suppose $\varepsilon > 0$. There is a $\delta > 0$ such that whenever 
	\begin{itemize}
\item[(a)] $\M$ is a Borel $\Lang$-structure that does not have the $(\delta, \brap)$-branching property and

\item[(b)] $g = \lceil 5 \cdot n_\Lang \cdot \brap \cdot \log \brap \rceil$,
\end{itemize}
there is a Borel $\M^+$ and an equitable partition $P$ of $\M^+$ such that 
\begin{itemize}
\item[(i)] $\dist(\M, \M^+) \leq \varepsilon$,
    
\item[(ii)] $P$ is indivisible with respect to $\M^+$, and
    
\item[(iii)] $|P| \leq 
	{(\frac{\varepsilon}{6 q_\Lang |\Lang|})}^{-(g+1)(g+2)}$.
\end{itemize}
\end{theorem}

\begin{proof}
	Let $\varepsilon_1 >0$, and
	let $\delta$ be as determined by \cref{Almost Stable Regularity for Relational Structures} (with $\varepsilon_1$ as its $\varepsilon$).

	Suppose $\M$ satisfies condition (a). Then there must be some $\delta_0 < \delta$ such that $\M$ also satisfies condition (a) with respect to $\delta_0$. Let $\varepsilon_0$ be such that $\delta_0 + |\Lang| \cdot (2^\brap \cdot q_\Lang)^{q_\Lang} \cdot \varepsilon_0 < \delta$ and $\varepsilon_0 < \varepsilon/3$. By \cref{Approximate Borel structures by finite structures structures in L1} (with $\varepsilon_0$ as its $\varepsilon$) we can find a finite $\H$ such that $\dist(\Borel{\H}, \M) < \varepsilon_0$.

	Suppose, towards a contradiction, that $\H$ has the $(\delta,\brap)$-branching property. Then there is some finite $\F$ that minimally has the $\brap$-branching property such that $\tind(\F, \H) \ge \delta$.
	By \cref{Subgraph densities as finite structures are bounded by the same as Borel Structures}, we then have $\tind(\F, \Borel{\H}) \ge \delta$.
	We also have
	$|\F|  \le 2^{\brap} q_\Lang$ 
	by \cref{Upper bound on size of structures which are minimal for branching}.
Then by \cref{Change in L1 doesn't change subgraph density}, we know that 
\[
	\bigl|\tind(\F, \Borel{\H}) - \tind(\F, \M)\bigr| \leq |\Lang| |\F|^{q_\Lang} \cdot \varepsilon_0 \le |\Lang| (2^{\brap} q_\Lang)^{q_\Lang} \cdot \varepsilon_0
\]
which implies that 
	\[
	\tind(\F, \M) \ge
	\delta - 
	|\Lang| (2^{\brap} q_\Lang)^{q_\Lang} \cdot \varepsilon_0
	> \delta_0.
		\]
		Hence $\M$ has the $(\delta_0,\brap)$-branching property,
		contradicting our choice of $\delta_0$.
Therefore $\H$ must not have the $(\delta,\brap)$-branching property.

	By \cref{Blow-ups give the same Borel structures},
we may replace $\H$ by a finite blow-up
	so that $|H|$ is large enough to apply \cref{Almost Stable Regularity for Relational Structures} (with $\frac{\varepsilon_1}{2}$ as its $\varepsilon$).
	We thereby obtain an $\Lang$-structure $\H^*$ with the same underlying set as $\H$ and an equitable partition $P_H$ such that 
\begin{itemize}
    \item $P_H$ is indivisible with respect to $\H^*$,
    
	\item $|P_H| \leq (\frac{\varepsilon_1}{2})^{-(g+1)(g+2)}$, and
    
    \item 
		$|E^\H \SymDiff E^{\H^*}| \leq \arity(E) 
		\cdot |H|^{\arity(E)} \cdot (\frac{\varepsilon_1}{2})$
		for all $E \in \Lang$.
\end{itemize}

But then by \cref{Difference in edges related to L1} we have 
\[
\dist(\Borel{\H}, \Borel{\H^*}) \leq \sum_{E \in \Lang} \arity(E) \cdot\varepsilon_1 \leq |\Lang| \cdot q_\Lang \cdot \varepsilon_1
\]
and so $\dist(\Borel{\H^*}, \M) \leq |\Lang| \cdot q_\Lang \cdot \varepsilon_1 + \varepsilon_0$.

We may similarly replace $\H^*$ by a finite blow-up so as to
	apply
\cref{Find equitable indivisible partitions on [0 1] from finite ones} to find a Borel $\Lang$-structure $M^+$ and an equitable partition $P$ such that 
\begin{itemize}
\item $|P| = |P_H|$,

\item $P$ is indivisible for $E^+$, and

\item $\dist(\M^+, \Borel{\H^*}) \leq |\Lang| \cdot \frac{|P|-1}{k} \leq \varepsilon_1$.  
\end{itemize}

	So we have $\dist(\M^+, \M) \leq \varepsilon_1 + |\Lang| \cdot q_\Lang \cdot
	\varepsilon_1
	+ \varepsilon_0$. 
Hence for $\varepsilon_1 \defas \frac{\varepsilon}{3 q_\Lang |\Lang|}$,
	we have $\dist(\M^+, \M) \leq \varepsilon$. Further, as $|P| = |P_H|$ we have $|P| \leq  
	{(\frac{\varepsilon_1}{2})}^{-(g+1)(g+2)}
	= {(\frac{\varepsilon}{6 q_\Lang |\Lang|})}^{-(g+1)(g+2)}$.
\end{proof}

\section*{Acknowledgements}

The authors would like to thank M.~Malliaris for helpful conversations.

The main results of this paper were presented at the 
North American Annual Meeting of the Association for Symbolic Logic \cite{ASLMeeting} at the University of Connecticut on May 26, 2016, and at the Workshop on Model Theory of Finite and Pseudofinite Structures\footnote{\url{http://www.maths.leeds.ac.uk/fps/programme.html}}
at the University of Leeds on July 29, 2016.

This research was facilitated by the 
the Trimester
Program on Universality and Homogeneity of the Hausdorff
Research Institute for Mathematics at the University of Bonn
(2013), the Program on Model Theory, Arithmetic Geometry and Number Theory
of the Mathematical Sciences Research Institute (2014),
and the Lorentz Center workshop on Logic and Random Graphs (2015).

Work on
this publication by C.\,F.\ was made possible through the support of
ARO grant W911NF-13-1-0212
and a grant from Google.


\newcommand{\etalchar}[1]{$^{#1}$}
\def\cprime{$'$} \def\polhk#1{\setbox0=\hbox{#1}{\ooalign{\hidewidth
  \lower1.5ex\hbox{`}\hidewidth\crcr\unhbox0}}}
  \def\polhk#1{\setbox0=\hbox{#1}{\ooalign{\hidewidth
  \lower1.5ex\hbox{`}\hidewidth\crcr\unhbox0}}} \def\cprime{$'$}
  \def\cprime{$'$} \def\cprime{$'$} \def\cprime{$'$} \def\cprime{$'$}
  \def\cprime{$'$} \def\cprime{$'$} \def\cprime{$'$} \def\cprime{$'$}
\providecommand{\bysame}{\leavevmode\hbox to3em{\hrulefill}\thinspace}
\providecommand{\MR}{\relax\ifhmode\unskip\space\fi MR }
\providecommand{\MRhref}[2]{%
  \href{http://www.ams.org/mathscinet-getitem?mr=#1}{#2}
}
\providecommand{\href}[2]{#2}

\end{document}